\documentclass[12pt,a4paper]{article}

\usepackage{amsfonts}
\newcommand{\R}{\mathbb{R}}

\usepackage{graphicx}
\usepackage{enumerate}
\usepackage{amsmath}
\usepackage{amssymb}
\usepackage{mathtools}
\usepackage{amsthm}
\usepackage{hyperref}
\usepackage{enumerate}

\usepackage{mathrsfs}
\usepackage[left=1in,right=1in, top=1.5in, bottom=1.5in]{geometry}

\usepackage[tiny]{titlesec}

\makeatletter
\def\cleardoublepage{\clearpage\if@twoside \ifodd\c@page\else
	\hbox{}
	\vspace*{\fill}
	\begin{center}
	\end{center}
	\vspace{\fill}
	\thispagestyle{empty}
	\newpage
	\if@twocolumn\hbox{}\newpage\fi\fi\fi}
\makeatother

\newtheorem{prop}{Proposition}
\newtheorem{theo}{Theorem}
\newtheorem{lemma}{Lemma}
\newtheorem{conj}{Conjecture}
\newtheorem{coro}{Corollary}

\title{Log-Brunn-Minkowski inequality under symmetry}
\author{K\'aroly J. B\"or\"oczky\thanks{This research was partially supported by National Research, Development and Innovation Office,  NKFI K 132002.}, Pavlos Kalantzopoulos}

\begin{document}
	\maketitle


\begin{abstract}
	We prove the log-Brunn-Minkowski conjecture for convex bodies with symmetries to $n$ independent hyperplanes, and discuss the equality case and the uniqueness of the solution of the related case of the logarithmic Minkowski problem. We also clarify a small gap in the known argument classifying the equality case of the log-Brunn-Minkowski conjecture for unconditional convex bodies.
\end{abstract}

\section{Introduction}

The classical Brunn-Minkowski inequality and the  Minkowski problem form the core of various areas in fully nonlinear partial differential equations, probability, additive combinatorics and convex geometry
(see Schneider \cite{book4}, Trudinger, X.-J. Wang \cite{wang} and Tao, Vu \cite{TaoVu}). 
For recent related work in the theory of valuations, algorithmic theory and the Gaussian setting, see 
Jochemko, Sanyal \cite{JS}, Kane \cite{kane}, Gardner, Gronchi \cite{GG} and Gardner, Zvavitch \cite{GZ}.
Extending it, Lutwak \cite{Lut1} initiated the rapidly developing new  $L_p$-Brunn-Minkowski theory.

The classical Minkowski's existence theorem due to Minkowski and Aleksandrov describes the so called surface area measure 
$S_K$ on $S^{n-1}$ (the case $p=1$)
 of a convex body $K$ in $\R^n$ where the regularity of the solution  is well investigated by
 Nirenberg \cite{NIR}, Cheng and Yau \cite{CY}, Pogorelov \cite{POG} and Caffarelli \cite{LC}.
After the first major results about the $L_p$-Brunn-Minkowski inequality and $L_p$-Minkowski problem for
 a range of $p$ by Firey \cite{Fir62}, Lutwak \cite{Lut1}, Chou, Wang \cite{CW} and 
Hug, Lutwak, Yang, Zhang \cite{HLYZ2}, the recent papers
Kolesnikov, Milman \cite{KolMilsupernew}, Milman \cite{EMilman},   
Bianchi, B\"or\"oczky, Colesanti, Yang \cite{BBCY},
Chen, Huang, Li \cite{CHL},  Ivaki \cite{Ivaki} present new developments. 

The cone volume measure or $L_0$ surface area measure $V_K$ on  $S^{n-1}$, originating from the papers
Firey \cite{Fir74} and Gromov and Milman \cite{GromovMilman},  has become an indispensable tool in the last decades (see  Barthe, Gu\'{e}don, Mendelson, Naor \cite{BG}, Naor \cite{N45},  Paouris, Werner \cite{PW}, B\"or\"oczky, Henk \cite{BoH16}).
If a convex body $K$ contains the origin, then its cone volume measure is
$dV_K=\frac1n\,h_K\,dS_K$ where $h_K$ is the support function of $K$ and $V_K(S^{n-1})=V(K)$ where the volume $V(K)$ is the $n$-dimensional Lebesgue measure of $K$. 
In particular, the classical formulation of the corresponding Monge-Amp\`ere equation "Logarithmic Minkowski problem" on the sphere $S^{n-1}$ is
\begin{equation}
\label{Mongehf}
h\det(\nabla^2 h+h\,{\rm Id})=  nf
\end{equation}
where $f$ is a given essentially positive function in $L_1(S^{n-1})$,  and the unknown $h$ on $S^{n-1}$ is the restriction of
the support function of a convex body containing the origin to $S^{n-1}$, and  $\nabla^2 h$ is the Hessian of $h$ with respect to a moving orthonormal frame. Following Firey \cite{Fir74} who first stated the Logarithmic Minkowski problem, we consider the Monge-Amp\`ere equation in the Alexandrov sense, namely, given a finite non-trivial Borel measure $\mu$ on $S^{n-1}$, we search for a convex body $K$ with $o\in K$ and
\begin{equation}
\label{MongeVK}
V_K=  \mu.
\end{equation}

Following partial results by Chou, Wang \cite{CW}, 
He, Leng,  Li \cite{HLL06},
Henk, Sch\"urman, Wills \cite{HSW}, Stancu \cite{Stancu,Stancu1},
Xiong \cite{Xio10}
the paper B\"or\"oczky, Lutwak, Yang, Zhang \cite{BLYZ13} characterized even cone volume measures
by the subspace concentration condition; namely, a finite non-trivial Borel even measure $\mu$ on $S^{n-1}$ is a cone volume measure
 if and only if
\begin{description}
\item[(i)] $\mu(L\cap S^{n-1})\leq \frac{{\rm dim}\,L}{n}\cdot\mu(S^{n-1})$ for any proper linear subspace $L$;
\item[(ii)] $\mu(L\cap S^{n-1})=\frac{{\rm dim}\,L}{n}\cdot\mu(S^{n-1})$ in (i)
 is equivalent with
${\rm supp}\,\mu\subset L\cup L^\bot$.
\end{description}

However, the  characterization of the cone volume measure of non-origin symmetric convex bodies is wide open.
While Chen, Li, Zhu \cite{CSL} have recently verified that the subspace concentration condition is sufficient to ensure that a measure on  $S^{n-1}$ is a cone volume measure, not even a meaningful conjecture is known about the right necessary conditions. 
All what is known concerning characterization is that  B\"or\"oczky, Heged\H{u}s \cite{BoH15} characterized the restriction of a cone volume measure to an antipodal pair of points.

 B\"or\"oczky, Lutwak, Yang, Zhang \cite{BLYZ12}
proposed the logarithmic Brunn-Minkowski conjecture Conjecture~\ref{logBMconj} in the even case. Before stating it, we
recall the Brunn-Minkowski inequality, which says that for every pair of convex bodies $K$ and $L$ in $\R^n$, and 
for every $\lambda\in(0,1)$, one has 
\begin{equation}\label{BM-power}
V((1-\lambda) K+\lambda L)^{\frac{1}{n}}\geq (1-\lambda) V(K)^{\frac{1}{n}}+\lambda V(L)^{\frac{1}{n}},
\end{equation}
with equality if and if $K$ and $L$ are homothetic (i.e. $L=\gamma K+x$ for $\gamma>0$, $x\in R^n$);
or, equivalently
\begin{equation}\label{BM-rst}
V((1-\lambda) K+\lambda L)\geq V(K)^{1-\lambda} V(L)^{\lambda}
\end{equation}
with equality if and only if $K$ and $L$ are translates.
See Gardner \cite{gardner} or Schneider \cite{book4} for more details. Analytically, the Brunn-Minkowski inequality has numerous realizations as a Poincare-type inequality (see, e.g. Colesanti \cite{Col1}, Colesanti, Hug, Saorin-Gomez \cite{Col2}, Colesanti, Livshyts, Marsiglietti \cite{CLM}, Kolesnikov, Milman \cite{KolMil}). This approach also led to strong stability versions of the Brunn-Minkowski inequality by Figalli, Maggi, Pratelli \cite{fig1} and Figalli, Jerison \cite{fig2}.
These analytic approaches stem from Hilbert's operator theoretic proof of the Brunn-Minkowski inequality from around 1900
(see Bonnesen, Fenchel \cite{BoF87}), which argument was further developed by Alexandrov in the 1930's, see Kolesnikov, Milman \cite{KolMil} for a modern presentation of their ideas.

If the shapes of $K$ and $L$ are substantially different, then the Brunn-Minkowski  inequality may provide a really bad estimate, which may be insufficient for certain applications. Below we describe a few conjectures  robustly strengthening the Brunn-Minkowski inequality.

For $\lambda\in(0,1)$, the geometric mean of the origin symmetric convex bodies $K$ and $L$ is 
\begin{equation}
(1-\lambda) \cdot K+_0 \lambda\cdot L:=\{x\in \R^n\,:\, \langle x,u\rangle\leq h_K^{1-\lambda}(u) h_L^{\lambda}(u)\,\,
\forall u\in S^{n-1}\}
\subset  (1-\lambda) K+\lambda L,
\end{equation}
where $\langle\cdot,\cdot\rangle$ is the standard scalar product in $\R^n$. The following possible strengthening of the Brunn-Minkowski inequality is widely known as logarithmic Brunn-Minkowski conjecture (see B\"or\"oczky, Lutwak, Yang, Zhang \cite{BLYZ12}
for the origin symmetric case).

\begin{conj}[log-Brunn-Minkowski Conjecture]
	\label{logBMconj} 
	For any pair $K$ and $L$ of convex bodies in $\R^n$, there exist 
	$z_K\in {\rm int}\,K$ and $z_L\in {\rm int}\,L$ such that
	for any $\lambda\in(0,1)$, we have 
	\begin{align}\label{log-B-M-I}
	V((1-\lambda)\cdot (K-z_K) +_0 \lambda\cdot (L-z_L))\geq V(K)^{1-\lambda} V(L)^\lambda
	\end{align}
	where $z_K=z_L=o$ if $K$ and $L$ are origin symmetric.
	In addition, equality holds if and only if $K=K_1+\ldots + K_m$ and $L=L_1+\ldots + L_m$ for compact convex sets 
	$K_1,\ldots, K_m,L_1,\ldots,L_m$ of dimension at least one where $\sum_{i=1}^m{\rm dim}\,K_i=n$
	and $K_i$ and $L_i$ are homothetic, $i=1,\ldots,m$.
\end{conj}

We note that the choice of the right translates is important in Conjecture~\ref{logBMconj} according to the examples by Nayar, Tkocz \cite{NaT}.

Conjecture \ref{logBMconj} was verified in the plane $\R^2$ by Xi, Leng \cite{XiL16}, $z_K$ and $z_L$ depend both 
on $K$ and $L$. However, one would conjecture that  $z_K$ and $z_L$ can be chosen to be the centroid of $K$ and $L$, 
but this stronger conjecture is open even in the plane. 

In the origin symmetric case, an equivalent formulation of Conjecture~\ref{logBMconj}  is the logarithmic Minkowski conjecture
according to B\"or\"oczky, Lutwak, Yang, Zhang  \cite{BLYZ12}.

\begin{conj}[Log-Minkowski conjecture]
	\label{Log-M}
	If $K$ and $L$ are origin symmetric convex bodies in $\R^n$, $n\geq 2$, then
	\begin{equation}\label{Log-M_eq}
	\int_{S^{n-1}}\log \frac{h_L}{h_K}\,dV_K\geq \frac{V(K)}n\log\frac{V(L)}{V(K)},
	\end{equation}
	with equality as in Conjecture~\ref{logBMconj}.
\end{conj}

The argument in \cite{BLYZ13} yields that for $o$-symmetric convex bodies with $C^\infty_+$ boundary, uniqueness of the
convex body with a prescribed cone volume measure 
 is equivalent to the log-Minkowski conjecture (see Section~\ref{secuniqueness}).
In particular,  uniqueness of the solution of the Monge-Amp\'ere equation \eqref{Mongehf} for any prescribed even positive $C^\infty$ function $f$ on $S^{n-1}$ implies 
the log-Brunn-Minkowski
and log-Minkowski conjectures (without the characterization of equality) for any 
$o$-symmetric convex bodies.

Let us summarize what is known about the Log-Brunn-Minkowski Conjecture~\ref{logBMconj} and 
Log-Minkowski Conjecture~\ref{Log-M}.
Concerning planar bodies, B\"or\"oczky, Lutwak, Yang, Zhang  \cite{BLYZ12} verified both conjectures
in origin symmetric case,  and Xi, Leng \cite{XiL16} proved Conjecture~\ref{logBMconj} in full generality. 
Turning to higher dimensions, besides the cases of unconditional convex bodies (see below) and
 complex bodies by Rotem \cite{Rotem}, these conjectures are proved when $K$ is close to be an ellipsoid in the sense of Hausdorff metric by a combination of the local estimates by Kolesnikov, Milman \cite{KolMilsupernew} and the use of the continuity method in PDE by Chen, Huang, Li \cite{CHL}. Another even more recent proof of this result based on Alexandrov's approach of considering the Hilbert-Brunn-Minkowski operator for polytopes and  \cite{KolMilsupernew}  is due to Putterman \cite{Put}. In addition, an isomorphic version of the Logarithmic Minkowski Problem is verified in Milman \cite{EMilman}.

We note that the conjectured uniqueness of the solution  of the Logarithmic, or $L_0$-Minkowski problem \eqref{Mongehf} for even positive $C^\infty$ function $f$ has a special role within the $L_p$-Minkowski Problems as if $p<0$, then
it is known that the even solution may not be unique (see
Jian, Lu, Wang \cite{JLW15},  Li, Liu, Lu \cite{LLL}, Milman \cite{Mil}).

We say that a set $X\subset \R^n$ is invariant under  $A\in{\rm GL}(n)$, if $AX=X$. Recall, a set $X$ is unconditional with respect to a fixed orthonormal basis $e_1,\ldots,e_n$ of $\R^n$ if it is symmetric through each coordinate hyperplane $e_i^\bot$; or in other words, $(x_1,\ldots,x_n)\in X$ implies that $(\pm x_1,\ldots,\pm x_n)\in X$. 

The Log-Brunn-Minkowski Conjecture~\ref{logBMconj} and 
Log-Minkowski Conjecture~\ref{Log-M} were verified for unconditional convex bodies 
(in a slightly stronger form for coordinatewise products, see the Appendix Section~\ref{secAppendix}) by several authors
like
Bollobas, Leader \cite{BoL95} and later indepently by 
Cordero-Erausquin, Fradelizi, Maurey \cite{CEFM04} even before the log-Brunn-Minkowski conjecture 
was stated, and the equality case was described by Saroglou \cite{Sar15}. Actually, the paper  \cite{Sar15} contains a small gap concerning the equality case, and we clarify the argument in the Appendix Section~\ref{secAppendix}.

We note that the arguments about the coordinatewise product of unconditional convex bodies all use the multiplicative form of the Pr\'ekopa-Leindler inequality stated by Ball \cite{Bal} (see also Theorem~\ref{PL}, Uhrin \cite{Uhr94} and Bollobas, Leader \cite{BoL95}). We write $K_1\oplus\ldots \oplus K_m$ to denote the Minkowski sum of compact convex sets $K_1,\ldots,K_m\subset\R^n$ if their affine hulls are pairwise orthogonal.

\begin{theo}[Bollobas-Leader, Uhrin, Saroglou]
\label{B.F.M.S.}
	If $K$ and $L$ are unconditional convex bodies in $\R^n$ with respect to the same orthonormal basis and $\lambda\in(0,1)$, then
	\begin{equation}
	\label{logBMeq}
	V((1-\lambda)\cdot K +_0 \lambda\cdot L)\geq V(K)^{1-\lambda} V(L)^\lambda.
	\end{equation}
	In addition, equality holds  if and only if $K=K_1\oplus\ldots \oplus K_m$ and $L=L_1\oplus\ldots \oplus L_m$ for unconditional compact convex sets 
	$K_1,\ldots, K_m,L_1,\ldots,L_m$ of dimension at least one where $K_i$ and $L_i$ are dilates, $i=1,\ldots,m$.
\end{theo}

We call a map $A\in{\rm GL}(n)$ a linear reflection if 
 $Ax=x$ for $x$ in a $(n-1)$-dimensional linear subspace $H$, $A\neq {\rm Id}_n$ and $A^2= {\rm Id}_n$ (see Davis \cite{Dav08}, Humphreys \cite{Hum90}, Vinberg \cite{Vin71}). In this case, $\det A=-1$
and there exists $u\in S^{n-1}\setminus H$ with $A(u)=-u$.
We observe that a linear reflection $A\in{\rm GL}(n)$ is a classical "orthogonal" reflection if and only if $A\in O(n)$.
In this paper, we show the log-Brunn-Minkowski Conjecture for pairs of convex bodies $K$ and $L$ that have a more general symmetry assumption than unconditional; namely, 
 when $K$ and $L$ are invariant under linear reflections $A_1,...,A_n\in{\rm GL}(n)$ which act identically on some $(n-1)$-dimensional linear subspaces $H_1,...,H_n$ such that $\cap_{i=1}^n H_i=\{o\}$.
We note that the symmetry assumption on $K$ and $L$ in Theorem~\ref{B.F.M.S.} is that for a fixed orthonormal basis $e_1,\cdots,e_n$, both $K$ and $L$ are invariant under the orthogonal reflections through $e_1^\bot,\ldots,e_n^\bot$.

\begin{theo}
	\label{logBMsymmetry}
	Let  $\lambda\in(0,1)$. If $A_1,\ldots,A_n$ are linear reflections such that $H_1\cap\ldots\cap H_n=\{o\}$ holds for the associated hyperplanes $H_1,\ldots,H_n$, and the convex bodies $K$ and $L$ are invariant under 
	$A_1,\ldots,A_n$, then
	$$
	V((1-\lambda)\cdot K +_0 \lambda\cdot L)\geq V(K)^{1-\lambda} V(L)^\lambda.
	$$
	In addition, equality holds  if and only if $K=K_1+\ldots + K_m$ and $L=L_1+\ldots + L_m$ for compact convex sets 
	$K_1,\ldots, K_m,L_1,\ldots,L_m$ of dimension at least one and invariant under $A_1,\ldots,A_n$ where $\sum_{i=1}^m{\rm dim}\,K_i=n$
	and $K_i$ and $L_i$ are homothetic, $i=1,\ldots,m$.
\end{theo}

The type of symmetry as in Theorem~\ref{logBMsymmetry}  has already occured in
Barthe and Fradelizi \cite{BaF13}, who proved the Mahler conjecture under the same symmetry assumptions,
and in Barthe, Cordero-Erausquin \cite{BaC13}, who bounded the isotropic constant. This project also builds on the approach of \cite{BaF13}.

Let us list various consequences of  Theorem~\ref{logBMsymmetry}.
We observe that Theorem~\ref{logBMsymmetry} settles the log-Brunn-Minkowski conjecture for convex bodies invariant under the symmetry group of a regular polytope. 

\begin{coro}
	\label{logBMsymmetryregular}
	Let  $\lambda\in(0,1)$. If $G$ is the group of symmetries of a regular polytope $P$ centered at the origin $o$ in $\R^n$, and the convex bodies $K$ and $L$ are invariant under 
	$G$, then
	$$
	V((1-\lambda)\cdot K +_0 \lambda\cdot L)\geq V(K)^{1-\lambda} V(L)^\lambda.
	$$
\end{coro}

According to B\"or\"oczky, Lutwak, Yang, Zhang  \cite{BLYZ12}, the $L_0$-sum is covariant under linear tranformations
(see also \eqref{linearinv} in Section~\ref{secLBM}). Thus for any subgroup $G\subset {\rm GL}(n,\R)$, if
$K$ and $L$ are convex bodies containing the origin in their interior and invariant under $G$, then the same holds for
$(1-\lambda)\cdot K +_0 \lambda\cdot L$ for any $\lambda\in(0,1)$. Therefore, 
Theorem~\ref{logBMsymmetry} and the method of \cite{BLYZ12} imply the following (see 
 Theorem~\ref{logMcalC} in Section~\ref{secuniqueness}).

\begin{theo}
	\label{logMsymmetry}
	Let  $\lambda\in(0,1)$. If $A_1,\ldots,A_n$ are linear reflections such that $H_1\cap\ldots\cap H_n=\{o\}$ holds for the associated hyperplanes $H_1,\ldots,H_n$, and the convex bodies $K$ and $L$ are invariant under 
	$A_1,\ldots,A_n$, then
$$
\int_{S^{n-1}}\log \frac{h_L}{h_K}\,dV_K\geq \frac{V(K)}n\log\frac{V(L)}{V(K)}
$$
	with equality as in Theorem~\ref{logBMsymmetry}.
\end{theo}
{\bf Remark } If in addition, $V(K)=V(L)$ in  Theorem~\ref{logMsymmetry}, then
$$
	\int_{S^{n-1}}\log h_L\,dV_K\geq \int_{S^{n-1}}\log h_K\,dV_K.
$$

Via the method of Saroglou \cite{Sar16} (see Section~\ref{secGaussian}), we obtain the analogue of
Theorem~\ref{logBMsymmetry} for the Gaussian measure $\gamma$ where
$d\gamma(x)=\frac1{(2\pi)^n}\exp(\frac{-\|x\|^2}2)\,dx$.

\begin{theo}
	\label{Gaussian-BMsymmetry}
	Let  $\lambda\in(0,1)$. If  $H_1\cap\ldots\cap H_n=\{o\}$ holds for the linear hyperplanes $H_1,\ldots,H_n$, and the convex bodies $K$ and $L$ are invariant under the orthogonal reflections through $H_1,\ldots,H_n$, then
	$$
	\gamma((1-\lambda)\cdot K +_0 \lambda\cdot L)\geq \gamma(K)^{1-\lambda} \gamma(L)^\lambda.
	$$
\end{theo}
\noindent{\bf Remark} Actually, Theorem~\ref{Gaussian-BMsymmetry} holds for any log-concave measure with rotationally symmetric density in place of the Gaussian density (see 
Theorem~\ref{Logconv-BMsymmetry} in Section~\ref{secGaussian}).\\

Theorem~\ref{logBMsymmetry} together with Proposition~1 in
Livshyts, Marsiglietti, Nayar, Zvavitch \cite{LMNZ}  immediately imply the following
concerning the conjecture of Gardner, Zvavitch \cite{GZ}.

\begin{theo}
	\label{GZ-BMsymmetry1}
	Let  $\lambda\in(0,1)$. If  $H_1\cap\ldots\cap H_n=\{o\}$ holds for the linear hyperplanes $H_1,\ldots,H_n$, and the convex bodies $K$ and $L$ are invariant under the orthogonal reflections through $H_1,\ldots,H_n$, then
$$
\gamma((1-\lambda) K + \lambda\, L)^{\frac1n}\geq (1-\lambda)\gamma(K)^{\frac1n}
+\lambda\,\gamma(L)^{\frac1n}.
$$
\end{theo}

We note that Eskenazis, Moschidis \cite{EsM21} proved the Gardner, Zvavitch conjecture in \cite{GZ} for origin symmetric convex bodies, and Kolesnikov, Livshyts \cite{KolLiv} obtained partial results concerning the Gardner-Zvavitch conjecture if the Gaussian centroids of the convex bodies $K$ and $L$ are the origin.

Concerning uniqueness of cone-volume measure, we have the following statement resulting from
Theorem~\ref{logMsymmetry} and the
method of B\"{o}r\"{o}czky, Lutwak, Yang, Zhang  \cite{BLYZ12} that was dealing with $o$-symmetric convex bodies
(see Section~\ref{secuniqueness}).

\begin{theo}
	\label{Minkowski-symmetry-uniqueness}
	Let   $A_1,\ldots,A_n$ be linear reflections such that $H_1\cap\ldots\cap H_n=\{o\}$ holds for the associated hyperplanes $H_1,\ldots,H_n$. For convex bodies $K$ and $L$ are invariant under 
	$A_1,\ldots,A_n$, we have $V_K=V_L$ if and only if $V(K)=V(L)$ and
	$K=K_1+\ldots + K_m$ and $L=L_1+\ldots + L_m$ for compact convex sets 
	$K_1,\ldots, K_m,L_1,\ldots,L_m$ of dimension at least one and invariant under $A_1,\ldots,A_n$ where 
	$\sum_{i=1}^m{\rm dim}\,K_i=n$
	and $K_i$ and $L_i$ are homothetic, $i=1,\ldots,m$.
\end{theo}

According to Chen, Li, Zhu \cite{CSL}, for general convex bodies, no analogue of 
Theorem~\ref{Minkowski-symmetry-uniqueness} can be expected, for example, $V_K=V_L$ may hold for two
non-homothetic convex bodies $K$ and $L$ with smooth boundary.

We note that for any convex body $K$, its centroid
$\frac{1}{V(K)}\int_Kx\,dx$ is invariant under any affine transformation which leaves $K$ invariant.
Therefore,  
Theorem~1.1 in B\"or\"oczky, Henk \cite{BoH16} and 
Theorem 1.4 in Bianchi, B\"or\"oczky, Colesanti, Yang \cite{BBCY} 
yield that the subspace concentration condition
characterizes the cone volume measures of convex bodies with high symmetry. 

\begin{theo}
\label{VKsymchar}
Let $G\subset O(n)$ be a group acting on $S^{n-1}$ without fixed points, and let
$\mu$ be a finite non-trivial Borel measure on $S^{n-1}$ invariant under $G$.
Then there exists a $G$ invariant solution
of the logarithmic Minkowski equation \eqref{MongeVK} in the Alexandrov sense
 if and only if
\begin{description}
\item[(i)] $\mu(L\cap S^{n-1})\leq \frac{{\rm dim}\,L}{n}\cdot\mu(S^{n-1})$ for any proper linear subspace $L$;
\item[(ii)] $\mu(L\cap S^{n-1})=\frac{{\rm dim}\,L}{n}\cdot\mu(S^{n-1})$ in (i)
 is equivalent with
${\rm supp}\,\mu\subset L\cup L^\bot$.
\end{description}
\end{theo}

Concerning the organization of the paper, after some preparation, we prove our main result Theorem~\ref{logBMsymmetry} in Section~\ref{secLBM}. Out of the consequences of Theorem~\ref{logBMsymmetry},
Section~\ref{secGaussian} discusses Theorem~\ref{Gaussian-BMsymmetry},  and its more
general version Theorem~\ref{Logconv-BMsymmetry}. In addition, Theorem~\ref{logMsymmetry}
and Theorem~\ref{Minkowski-symmetry-uniqueness} are verified in Section~\ref{secuniqueness}.
Finally, we discuss the proof of Theorem~\ref{B.F.M.S.} and the small correction in the argument characterizing the equality case in the Appendix Section~\ref{secAppendix}.

\section{Some properties of convex compact sets and the $L_0$-sum}

In this section, we collect some known facts about convex bodies. For some of these statements we have not found a reference, and in these cases, we provide the short arguments.

For notions in convexity, see Schneider \cite{book4}.
In the Euclidean  $n$-space $\R^n$, we denote the standard inner product by $\langle\cdot,\cdot\rangle$,  
the Euclidean norm by $|\cdot|$,  the volume (Lebesgue measure) by $V(\cdot)$, and  the $k$-dimensional Hausdorff measure
by $\mathcal{H}^k$. We denote the Euclidean unit ball centered at the origin and sphere in $\R^n$ by $B^n_2$ and 
$S^{n-1}=\partial B^n_2$, respectively.  For a convex subset $K$ of $\R^n$,
we write  $\partial K$ and  ${\rm relint}K$ to denote the relative boundary and the relative interior with respect to the affine hull of $K$, respectively, and
${\rm int}K$ the interior of $K$ with respect to $\R^n$. A convex body is a compact convex subset with non-empty interior.
We write $M|E$ to denote the orthogonal projection of a compact convex set $M$ into a linear subspace $E$ in $\R^n$.

The \textit{support function} $h_K:\R^n\to \R$ of a compact convex set $K$ in $\R^n$ is defined, for $x\in\R^n$, by
\begin{align*}
h_K(x)={\rm max}\{\langle x,y\rangle:y\in K\}.
\end{align*}
Note that support functions are positively homogeneous of degree one and subadditive. 
A vector $u\in\R^n\backslash\{o\}$ is an exterior normal vector at a  boundary point $x\in\partial K$ if
\begin{align*}
\langle x,u\rangle =h_K(u),
\end{align*} 
and it is called a unit exterior normal if $u\in S^{n-1}$.
A point $x\in\partial K$ is called \textit{smooth boundary point} if there exists a unique exterior unit  normal. We denote by $\partial' K$ the set of all smooth boundary points. It is well known that the set of all non-smooth boundary points of a convex body has 
$\mathcal{H}^{n-1}$-measure equal to $0$. 
The \textit{spherical image map} 
\begin{align*}
\nu_K:\partial'K\to S^{n-1}
\end{align*} 
sends every smooth boundary point to its unique outer unit normal. 

We recall that for $\lambda\in(0,1)$, the $L_0$ sum of two convex bodies $K$ and $L$ in $\R^n$ with $o\in{\rm int}\,K$ and $o\in{\rm int}\,L$
is the Wulff shape
\begin{eqnarray*}
(1-\lambda)\cdot K+_0\lambda \cdot L&=&\left\{x\in\R^n:\,\langle x,u\rangle\leq h_K(u)^{1-\lambda}h_L(u)^\lambda
\mbox{ }\forall u\in S^{n-1} \right\}\\
&=&\left\{x\in\R^n:\,\langle x,u\rangle\leq h_K(u)^{1-\lambda}h_L(u)^\lambda
\mbox{ }\forall u\in \R^n \right\}.
\end{eqnarray*}
We claim that for any $u\in\R^n\backslash \{o\}$, 
\begin{equation}
\label{L0sumsmooth}
\begin{array}{c}
\mbox{if $u$ is an exterior normal at a $z\in\partial' ((1-\lambda)\cdot K+_0\lambda \cdot L)$, then }\\
 h_{(1-\lambda)\cdot K+_0\lambda\cdot L}(u)=h_K(u)^{1-\lambda}h_L(u)^\lambda.
\end{array}
\end{equation}
We may assume that $u\in S^{n-1}$. Since $z\in\partial((1-\lambda)\cdot K+_0\lambda \cdot L)$ is a boundary point and $h_K$ and $h_L$ are
continuous, there exists some
$v\in S^{n-1}$ such that $h_{(1-\lambda)\cdot K+_0\lambda\cdot L}(v)=\langle z,v\rangle=h_K(v)^{1-\lambda}h_L(v)^\lambda$. However, $z$ is a smooth boundary point where there exists only a unique exterior unit normal; therefore, we have $u=v$, verifying \eqref{L0sumsmooth}.

We note that the logarithmic sum is linear covariant; namely, if $\Phi\in {\rm GL}(n,\R)$, then
\begin{align}
\label{linearinv}
\Phi[(1-\lambda)\cdot K +_0 \lambda\cdot L]=(1-\lambda)\cdot \Phi(K) +_0 \lambda\cdot \Phi(L).
\end{align}
This is based on the fact that $h_{\Phi K}(u)=h_K(\Phi^tu)$. Therefore, if $K$ and $L$ are two convex bodies in $\R^n$ invariant under some subgroup $G\subset {\rm GL}(n)$, then $(1-\lambda)\cdot K+_0\lambda\cdot L$ is
also invariant under $G$.

As the equality case of Theorem~\ref{logBMsymmetry}  indicates, we need a better understanding of convex bodies
that are sums convex compact sets in complementary linear subspaces.

\begin{lemma}[Folklore]
\label{SKsumofmany}
Let $K$ be a convex body in $\R^n$, and let $\xi_1,\ldots,\xi_m$, $m\geq 2$  be non-trivial complementary linear subspaces which together span $\R^n$.
Then $\nu_K( \partial' K)\subset \xi_1\cup\ldots\cup \xi_m$ if and only if there exist compact convex sets 
$K_1,\ldots,K_m$ with ${\rm lin}(K_i-K_i)=(\sum_{j\neq i}\xi_j)^\bot$ (and hence
${\rm dim}\,K_i={\rm dim}\,\xi_i$)  for $i=1,\ldots,m$ such that $K=K_1+\ldots+ K_m$.
\end{lemma}
\noindent{\bf Remark } If $K$ is unconditional and $K_1,\ldots,K_m$ are unconditional, then $K_i\subset \xi_i$,
$i=1,\ldots,m$.
\begin{proof}
We may assume that $o\in K$, and hence also that $o\in K_i$ for $i=1,\ldots,m$ if suitable $K_1,\ldots,K_m$ exists.

If $K=K_1+\ldots+ K_m$ for some  compact convex $K_i\subset (\sum_{j\neq i}\xi_j)^\bot$, $i=1,\ldots,m$, then
$$
\partial'K=\bigcup_{i=1}^m \left(\partial' K_i+\sum_{j\neq i}{\rm relint}\,K_j\right),
$$
which in turn yields that $\nu_K( \partial' K)\subset \xi_1\cup\ldots\cup \xi_m$ by the property
$$
\xi_i^\bot={\rm lin}\sum_{j\neq i}{\rm relint}\,K_j
$$
for $i=1,\ldots,m$ (here $\partial' K_i$ is the family of smooth points of the relative boundary of $K_i$).

On the other hand, let us assume that $\nu_K( \partial' K)\subset \xi_1\cup\ldots\cup \xi_m$, and let
$V_i=(\sum_{j\neq i}\xi_j)^\bot$.
For any $i=1,\ldots,m$, let us consider the convex compact set
$$
K_i=\left\{x\in V_i:\,\langle u,x\rangle\leq h_K(u) \mbox{ for all }u\in \xi_i\cap\nu_K( \partial' K) \right\}.
$$
As $V_i^\bot+\xi_i=\R^n$ and $V_i^\bot\cap\xi_i=\{o\}$ for $i=1,\ldots,m$, we deduce that $K_i$ is a 
${\rm dim}V_i={\rm dim}\xi_i$ dimensional compact convex set.
Since $K$ is the intersection of the supporting halfspaces at the smooth boundary points according to
Theorem~2.2.6 in Schneider \cite{book4}, the condition
$\nu_K( \partial' K)\subset \xi_1\cup\ldots\cup \xi_m$ implies
$$
K=\bigcap_{i=1}^m\left\{x\in\R^n:\langle u,x\rangle\leq h_K(u)\,\forall u\in\nu_K( \partial' K)\cap \xi_i\right\}=
\bigcap_{i=1}^m\left(K_i+\xi_i^\bot\right)=K_1+\ldots+K_m.
$$
\end{proof}

Next we show that equality really holds in Theorem~\ref{logBMsymmetry} when promised
(even without symmetry assumption).

\begin{lemma}[Folklore]
\label{sumlowerlogBMequa}
If $\lambda\in(0,1)$, $K$ and $L$ are convex bodies with $o\in{\rm int}K$ and $o\in{\rm int}L$,
and $K=K_1+\cdots+ K_m$ and $L=L_1+\cdots + L_m$ for 
$m\geq 1$ and compact convex sets $K_i,L_i$, $i=1,\cdots,m$, 
having dimension at least one and satisfying $o\in K_i$, $K_i=\theta_iL_i$ for $\theta_i>0$ for 
$i=1,\cdots,m$, and  $\sum_{i=1}^m{\rm dim}K_i=n$, then
\begin{enumerate}[(i)]
\item $(1-\lambda)\cdot K+_0\lambda\cdot L=\theta_1^\lambda K_1+\cdots+ \theta_n^\lambda K_m$;
\item $V((1-\lambda)\cdot K+_0\lambda\cdot L)=V(K)^{1-\lambda}V(L)^\lambda$.
\end{enumerate}
\end{lemma}
\begin{proof}
For $i=1,\ldots,m$, we write $V_i={\rm lin}K_i$, and $\xi_i=\left(\sum_{j\neq i} V_j\right)^\bot$.
We observe that if $u\in\xi_i\cap S^{n-1}$, then 
$$
h_K(u)=h_{K_i}(u)\mbox{ \ and \ }h_L(u)=\theta_ih_{K_i}(u).
$$
It follows from Lemma~\ref{SKsumofmany} that
\begin{eqnarray*}
K&=&\bigcap_{i=1}^m\left\{x\in\R^n:\,\langle u,x\rangle\leq h_K(u)\;\forall u\in\xi_i\cap S^{n-1}\right\}\\
L&=&\bigcap_{i=1}^m\left\{x\in\R^n:\,\langle u,x\rangle\leq \theta_ih_K(u)\;\forall u\in\xi_i\cap S^{n-1}\right\};
\end{eqnarray*}
therefore, $h_K(u)^{1-\lambda}\Big(\theta_ih_K(u)\Big)^\lambda=\theta_i^\lambda h_K(u)$ 
for $u\in\xi_i\cap S^{n-1}$ and $i=1,\ldots,m$ yields that
$$
(1-\lambda)\cdot K+_0\lambda\cdot L\subset 
\bigcap_{i=1}^m\left\{x\in\R^n:\,\langle u,x\rangle\leq \theta_i^\lambda h_K(u)\;\forall u\in\xi_i\cap S^{n-1}\right\}
=\sum_{i=1}^m\theta_i^\lambda K_i.
$$
To prove $\sum_{i=1}^m\theta_i^\lambda K_i\subset (1-\lambda)\cdot K+_0\lambda\cdot L$, it is enough to verify
\begin{equation}
\label{suminL0sum}
\sum_{i=1}^m\theta_i^\lambda h_{K_i}(u)\leq
 h_K(u)^{1-\lambda}h_L(u)^\lambda=\left(\sum_{i=1}^m h_{K_i}(u)\right)^{1-\lambda}
\left(\sum_{i=1}^m\theta_i h_{K_i}(u)\right)^{\lambda}
\end{equation}
for any $u\in S^{n-1}$. However, \eqref{suminL0sum} is a direct consequence of the H\"older inequality, completing the proof of (i).

We observe that setting $d_i={\rm dim}\,K_i$ for $i=1,\ldots,m$, we have
 $V(L)=\left(\prod_{i=1}^m\theta_i^{d_i}\right)V(K)$, and (i) yields that
$$
V((1-\lambda)\cdot K+_0\lambda\cdot L)=\left(\prod_{i=1}^m\theta_i^{d_i}\right)^\lambda V(K),
$$
verifying (ii).
\end{proof}

\section{Simplicial cones and Representation of Coxeter groups} 

We say that a convex subset $C\subset\R^n$ is a convex cone if $\lambda\,x\in C$ for any
$x\in C$ and $\lambda\geq 0$. The positive dual cone of $C$ is
$$
C^*=\{x\in\R^n:\,\langle x,y\rangle\geq 0 \mbox{ for each }y\in C\}.
$$
For any $n$ independent vectors $u_1,\cdots,u_n\in \R^n$, the convex cone $C$ generated by their positive hull
\begin{align}\label{simplisial convex cone form1}
C={\rm pos}\{u_1,\cdots,u_n\}=\left\{\sum_{i=1}^n\lambda_iu_i:\, \forall\; \lambda_i\geq 0\right\}
\end{align}
is called \textit{simplicial convex cone}. In this case, the positive dual cone is 
$$
C^*={\rm pos}\{u_1^*,\ldots,u_n^*\}
$$
where $\langle u_i,u_j^*\rangle=0$ if $i\neq j$ and $\langle u_i,u_i^*\rangle>0$. For $i=1,\cdots,n$, the \textit{facets} of $C$ are  
\begin{align*}
F_i={\rm pos}\{\{u_1,\cdots,u_n\}\setminus\{u_i\}\}=C\cap (u_i^*)^\bot,
\end{align*}
and the \textit{walls} of $C$ are the linear subspaces
\begin{align*}
W_i={\rm lin}\{\{u_1,\cdots,u_n\}\setminus\{u_i\}\}.
\end{align*}
Note that the orthogonal reflection ${\rm Ref}_{W_i}$ through the wall $W_i$ of $C$ is the map
$x\mapsto x-2\langle x,u_i^*\rangle u_i^*$.
We observe that $-u_i^*$  is an exterior normal to $F_i$, and 
\begin{align}\label{simplisial convex cone form2}
C=\{x\in\R^n: \langle x,u_i^*\rangle\geq 0 \text{ for }i=1,\cdots,n \}.
\end{align}

A linear subspace $E$ of $\R^n$ is called non-trivial if ${\rm dim}\,E\geq 1$. In this case, we write
$\mathcal{O}(E)$ to denote the group of orthogonal transformations of $E$ where $O(n)=\mathcal{O}(\R^n)$.
If $G$ is a group generated by reflections through $n$ independent linear hyperplanes $H_1,\ldots,H_n$ 
($n$ hyperplanes $H_1,\ldots,H_n$ with $H_1\cap\ldots\cap H_n=\{o\}$), and $H_i=v_i^\bot$ for $v_i\in \R^n\backslash\{o\}$ and 
$i=1,\ldots,n$, then any non-trivial $G$ invariant linear subspace $E$ is of the form 
\begin{equation}
E={\rm lin}I\mbox{ for non-empty }I\subset\{v_1,\ldots,v_n\}\mbox{ where $\langle v_i,v_j\rangle=0$
if $v_i\in I$ and $v_j\not\in I$.}
\end{equation}
We call an invariant linear subspace irreducible with respect to the action of $G$ if it has no proper $G$-invariant linear subspace.
It follows that there exist only finitely many irreducible subspaces  
$E_1,\ldots,E_k$, $k\geq 1$,
satisfying that 
\begin{itemize}
\item $\R^n=\oplus_{i=1}^kE_i$; 
\item $E_i$ and $E_j$ are orthogonal for $i\neq j$;
\item $G=G_1\times \ldots \times G_k$ where $G_i\subset O(E_i)$ acts irreducibly on $E_i$.  
\end{itemize}
This decomposition corresponds to the irreducible representations coming from the action of the closure of $G$ in $O(n)$, see
Humphreys \cite{Hum78} for representations of compact groups.

Typical example for a finite group $G\subset O(n)$ generated by reflections through $n$ independent hyperplanes and acting irreducibly on $\R^n$ is the symmetry group of a regular polytope $P$ in $\R^n$ whose centroid is the origin
(see McCammond \cite{McC} or Humphreys \cite{Hum90}). For example, if $P$ is a regular simplex, then the $n$ independent hyperplanes might be the perpendicular bisectors of the $n$ edges meeting at a fixed vertex of $P$.

The following Lemma~\ref{prop.for.chambers} defines the Weyl chamber associated to an irreducible action of a finite Coxeter group, and dicusses the fundamental properties. These Weyl chambers partition $\R^n$ into simplicial cones (see Lemma~\ref{prop.for.chambers} (ii) and (iii)).

\begin{lemma}[Coxeter]
\label{prop.for.chambers}
Let $G$ be a finite group generated by reflections through $n$ hyperplanes $H_1,\ldots,H_n$ with $H_1\cap\ldots,\cap H_n=\{o\}$
and acting irreducibly on $\R^n$.
Then there exists a simplicial cone $C={\rm pos}\{u_1,\cdots,u_n\}$ (called a Weyl chamber) such that
	\begin{enumerate}[(i)]
		\item \label{csjjbuicjksdc} the $n$ reflections through the walls of $C$ generate $G$;	
\item $\R^n=\cup_{g\in G} gC$;
\item if $ gC  \cap  {\rm int}C\neq \emptyset$ for some $g\in G$, then $g$ is the identity;
\item  $\langle x,y \rangle \geq0$ for $x,y\in C$ and writing $C^*={\rm pos}\{v_1,\ldots,v_n\}$, we have
$\langle v_i,v_j\rangle\leq 0$ provided $i\neq j$;
		\label{aytosigouraden}
\item for any partition  $\{1,\ldots,n\}=I\cup J$ with $I,J\neq\emptyset$ and $I\cap J=\emptyset$, there exist
$i\in I$ and $j\in J$ such that $\langle v_i,v_j\rangle<0$.
	\end{enumerate}
\end{lemma}

\begin{proof} 
According to the classical theory (see Humphreys \cite{Hum90}), one associates a so called root system to $G$; namely, a finite set $\Phi$ of non-zero vectors such that any  two are either independent or opposite, and the set of reflections in $G$ coincides with the set reflections through the linear $(n-1)$-dimensional subspaces orthogonal to the elements of $\Phi$. 
It is a well-known result (see Humphreys \cite{Hum90}) that there exists
some $n$ independent roots
 $v_1,\ldots,v_n\in \Phi$ such that any other root can be written as a linear combination of  $v_1,\ldots,v_n$ with all non-positive or all non-negative coefficients. Then $v_1,\ldots,v_n\in \Phi$ are called simple roots, and the simplicial cone 
$C=\{x\in\R^n:\,\langle x,v_i\rangle \geq 0\mbox{ for }i=1,\ldots,n\}$ satisfies  
(i), (ii), (iii); moreover, $v_1,\ldots,v_n$ satisfy that $\langle v_i,v_j\rangle\leq 0$ for $i\neq j$
(see Humphreys \cite{Hum90}), verifying the second half of (iv).

We complete the proof of (iv) by contradiction, so we suppose that there exist 
$x,y\in C$ satisfying $\langle x,y\rangle< 0$, and seek a contradiction.
We set $v_{n+1}=-x$ and $v_{n+2}=-y$; therefore, 
$\langle v_i,v_j\rangle\leq 0$ for $i,j=1,\ldots,n+2$ and
$\langle v_{n+1},v_{n+2}\rangle< 0$. According to Radon's theorem, there exist non-empty 
$A,B\subset \{1,\ldots,n+2\}$ with $A\cap B=\emptyset$, and $\alpha_i>0$ and $\beta_j>0$
for $i\in A$ and $j\in B$ such that
$$
\sum_{i\in A}\alpha_iv_i=\sum_{j\in B}\beta_jv_j=w.
$$
We deduce that
$$
0\leq \langle w,w\rangle=\sum_{i\in A}\sum_{j\in B}\alpha_i\beta_j\langle v_i,v_j\rangle,
$$
thus $\langle v_i,v_j\rangle\leq 0$ for $i\neq j$ yields that 
\begin{equation}
\label{viAvjB}
\mbox{$\langle v_i,v_j\rangle=0$ for $i\in A$ and $j\in B$,}
\end{equation}
and hence $w=o$. 
In turn, 
the independence of $v_1,\ldots,v_n$ shows that $A\cap\{v_{n+1},v_{n+2}\}\neq \emptyset$
and $B\cap\{v_{n+1},v_{n+2}\}\neq \emptyset$, which facts contradict
$\langle v_{n+1},v_{n+2}\rangle< 0$ by \eqref{viAvjB}.

Finally, we prove (v) again by contradiction. We suppose that there exists
a partition  $\{1,\ldots,n\}=I\cup J$ with $I,J\neq\emptyset$ and $I\cap J=\emptyset$ such that
$\langle v_i,v_j\rangle=0$ for
$i\in I$ and $j\in J$ (note that $\langle v_i,v_j\rangle\geq 0$ by (iv)). Then both ${\rm lin} \{v_i:\,i\in I\}$ and  
${\rm lin}\{v_j:\,j\in J\}$ are invariant under reflections through the walls of $C$, which contradicts the irreducibility of the action of $G$ on $\R^n$. 
\end{proof}

The main goal of this section is to prove the following statement which describes how the Weyl chambers essentially partitioning 
$\R^n$ 
 (see Proposition~\ref{basikothe} (iii)) are related to the group action.

\begin{prop}
\label{basikothe}	
	Let $G\subset O(n)$ be the closure of a 
group generated by 	the orthogonal reflections  through the hyperplanes
$H_1,\ldots,H_n$ of $\R^n$ with $H_1\cap\ldots\cap H_n=\{o\}$, let $E_1,\ldots,E_k$ be the corresponding irreducible subspaces. Then there exist an $n$-dimensional simplicial convex cone 
$C=\oplus_{\alpha=1}^k C_\alpha$ in $\R^n$ where $C_\alpha\subset E_\alpha$
is a Weyl chamber for the irreducible action of a 
finite subgroup 
$\widetilde{G}_\alpha\subset\mathcal{O}(E_\alpha)$  on $E_\alpha$ and 
$\widetilde{G}_\alpha$ is generated by reflections through the
walls of $C_\alpha$ in $E_\alpha$ for $\alpha=1,\ldots,k$.
In addition, 
\begin{enumerate}
\item $\widetilde{G}=\widetilde{G}_1\times\ldots\times\widetilde{G}_k$ is a subgroup of $G$;

\item  \label{the1} writing $W_1,\ldots,W_n$ to denote the walls of $C$, 
the reflections ${\rm Ref}_{W_\alpha}$, $\alpha=1,\cdots,n$, generate $\widetilde{G}$;

\item  \label{the2} $gC\cap{\rm int}\,C\neq\emptyset$ for $g\in \widetilde{G}$ implies that $g$ is the identity, and
		\begin{align*}
		\R^n=\bigcup_{g\in \widetilde{G}}gC;
		\end{align*}
\item if $C^*={\rm pos}\{v_1,\ldots,v_n\}$, then 
$\langle v_i,v_j\rangle\leq 0$ provided $i\neq j$;
\item If $K$ is a convex body in $\R^n$ invariant under $G$,
then $\nu_K(x)\in C$ for $x\in \partial'K\cap C$, and if moreover $\Phi\in GL(n)$ satisfies $\Phi( C)=\R^n_+$,
then the unconditional set  $\bar{K}$  defined by 
$\bar{K}\cap\R^n_+=\Phi(K\cap C)$ is an unconditional convex body\label{the3}.
	\end{enumerate} 
\end{prop}

We prepare the proof of Proposition~\ref{basikothe} with a series of lemmas mostly discussing well-known statements.

The following statement is Lemma~19 in Barthe, Fradelizi \cite{BaF13}.

\begin{lemma}[Barthe, Fradelizi]
\label{infinite-group}
If $G\subset O(n)$ is an infinite subgroup generated by reflections through $n$ hyperplanes $H_1,\ldots,H_n$ with $H_1\cap\ldots,H_n=\{o\}$, and  $G$ acts irreducibly on $\R^n$, then the closure of $G$ is $O(n)$. 
\end{lemma}

\begin{lemma}\label{Propert.H,C}
For $k\geq 2$, let $E_\alpha$, $\alpha=1,\ldots,k$ be pairwise orthogonal non-trivial linear subspaces of $\R^n$
with $\oplus_{\alpha=1}^kE_\alpha=\R^n$, and for $\alpha=1,\ldots,k$, let $G_\alpha\subset \mathcal{O}(E_\alpha)$ be a finite subgroup generated by 
reflections through ${\rm dim} E_\alpha$ independent hyperplanes of $E_\alpha$, and let $C_\alpha$ be a Weyl chamber for the action 
of $G_\alpha$. Then for the subgroup $G=G_1\times\ldots\times G_k$ of $O(n)$ and $C=\oplus_{\alpha=1}^kC_\alpha$,
we have
\begin{enumerate}[(i)]
\item $G$ is generated by the reflections through the walls of $C$;
\item $\cup \{gC:g\in G\}=\R^n$;
		\item if ${\rm int} gC \cap {\rm int} C\neq \emptyset$ for a $g\in G$, then $g$ is the identity; 
		\item $\langle x,y \rangle \ge 0$ for $x,y\in C$.
	\end{enumerate}
\end{lemma}
\begin{proof}
	\begin{enumerate}[(i)]
		\item  As $G=G_1\times\ldots\times G_k$ and $E_1,\ldots,E_k$ are pairwise orthogonal, 
Lemma~\ref{prop.for.chambers} (i) yields that
a set generators of $G$ is the $n$ reflections through the hyperplanes of 
$\R^n$ of the form $W+E_\alpha^\bot$ where for some $E_\alpha$, $\alpha=1,\ldots,k$, $W$ is a wall of $C_\alpha$ in 
$E_\alpha$. Since these $n$ hyperplanes of 
$\R^n$ are exactly the walls of $C$, we deduce (i).

\item Write $x\in \R^n$ as $x=x_1+\dots
		+x_k$ where $x_{\alpha}\in E_{\alpha}$, $\alpha=1,\ldots,k$. 
According to Lemma~\ref{prop.for.chambers} (ii), there exits $g_{\alpha}\in G_\alpha$ such that 
$x_{\alpha} \in g_{\alpha} C_{\alpha}$ for each $\alpha=1,\ldots,k$. 
Therefore $x\in gC$ for $g=(g_1,\dots,g_k)\in G$.
		\item Assume ${\rm int} gC \cap {\rm int} C\neq \emptyset$ for $g=(g_1,\cdots,g_k)\in G$. 
Projecting into each $E_\alpha$ shows that the relative interiors of $g_\alpha C_\alpha$ and $C_\alpha$ intersect for 
$\alpha=1,\ldots,k$; therefore,  $g_\alpha C_\alpha=C_\alpha$ for $\alpha=1,\ldots,k$ by 
Lemma~\ref{prop.for.chambers} (iii),
and hence  $gC=C$.
		\item This follows from Lemma~\ref{prop.for.chambers} \eqref{aytosigouraden} and the fact that 
the subspaces $E_1,\ldots,E_k$ are pairwise orthogonal.
	\end{enumerate}
\end{proof}

\begin{lemma}
\label{autokainden}
If $K$ is a convex body in $\R^n$, and there is  a simplicial convex cone $C$ such that $K$ is invariant with respect to the orthogonal reflections through the walls of $C$, then 
\begin{description}
\item{(i)} $\nu_K(z)\in C$ holds for any $z\in \partial'K\cap C$;
\item{(ii)} $\begin{array}{rcl}
K\cap C&=&\left\{x\in C:\langle x,\nu_K(z)\rangle \leq h_{K}(\nu_K(z))\mbox{ }\forall\,z\in\partial'K\cap C\right\}\\
&=&
\left\{x\in C:\langle x,u\rangle \leq h_{K}(u)\mbox{ }\forall\,u\in C\right\}.
\end{array}$
\end{description}
\end{lemma}
\begin{proof}
As in \eqref{simplisial convex cone form1} and \eqref{simplisial convex cone form2},
we write the cone $C$ as $C={\rm pos}\{u_1,\ldots,u_n\}$
and $C=\{z\in\R^n:\,\langle z,x_j\rangle\leq 0,\;j=1,\ldots,n\}$ for independent  
$u_1,\ldots,u_n\in S^{n-1}$ and $x_1,\ldots,x_n\in S^{n-1}$ satisfying
$\langle x_j,u_i\rangle =0$ for $j\neq i$ and $\langle x_j,u_j\rangle <0$ for $j=1,\ldots,n$.

For $z\in\partial'K\cap C$ and $j\in\{1,\cdots,n\}$,
 we show that 
$$
\langle \nu_K(z),x_j\rangle\leq0. 
$$
We use that $K$ is symmetric with respect to the wall $W_j:={\rm lin}\{u_1,\cdots,u_n\}\setminus\{u_j\}=x_j^\bot$  of $C$; or in other words, ${\rm Ref}_{W_j}K= K$.

If $z\in x_j^\bot$, then the symmetry of $K$ through $W_j$ shows that both $\nu_K(z)$ and ${\rm Ref}_{W_j}(\nu_K(z))$
are exterior normals at $z$, and hence $\nu_K(z)\in x_j^\bot$.

Therefore, let $\langle z,x_j\rangle<0$. As $\nu_K(z)$ is an exterior normal at $z$ and ${\rm Ref}_{W_j}z\in K$, we deduce that
$\langle \nu_K(z),({\rm Ref}_{W_j}z)-z\rangle\leq0$. However, $({\rm Ref}_{W_j}z)-z$ is a positive multiple of $x_j$,
thus $\langle \nu_K(z),x_j\rangle\leq 0$, which implies $\nu_K(z)\in C$ since $j$ was arbitrary. 

Since $K$ is the intersection of the supporting halfspaces at the smooth boundary points according to
Theorem~2.2.6 in Schneider \cite{book4}, (i) yields (ii).
	\end{proof}

\begin{lemma}
\label{unco.convex}
	Let $K$ be a convex body in $\R^n$  and let $C$ be a simplicial convex cone  such that $\langle x,y \rangle \geq 0$ for every $x,y\in C$ and $K$ is invariant with respect to the orthogonal reflections through the walls of $C$. If $\Phi\in GL(n)$ satisfies $\Phi C=\R^n_+$, then 
$\Phi^{-t}C\subset \R^n_+$ and
the unconditional set $\bar{K}$ defined by  $\bar{K}\cap\R^n_+=\Phi(K\cap C)$ is an unconditional convex body.
\end{lemma}

\begin{proof}
	To show the convexity of $\bar{K}$, we observe that 
$C\subset C^*$ holds for the  positive dual cone $C^*$ by the condition on $C$, and hence
\begin{align}\label{csjdnjsd}
\Phi^{-t}C\subset  \Phi^{-t}C^*=(\Phi C)^*=(\R^n_+)^*=\R^n_+.
\end{align}
Now if $z\in \partial'\Phi(K)\cap\R^n_+$, then $z=\Phi y$ for some  $y\in \partial' K \cap C$
where $\nu_{K}(y)\in C$ according to Lemma~\ref{autokainden}. 
Since $\Phi^{-t}\nu_{K}(y)$ is an exterior normal to $\partial\Phi(K)$ at $z=\Phi y$,
we conclude from 
\eqref{csjdnjsd} and the conditions on $C$ and $K$ that
\begin{align}\label{den exw idea pws na to pw}
z\in \partial'\Phi(K)\cap\R^n_+ \Rightarrow \nu_{\Phi(K)}(z)\in\R^n_+
\end{align}	

As $\bar{K}$ is an unconditional set and $\bar{K}\cap\R^n_+=\Phi(K)\cap\R^n_+$, 
its convexity is equivalent with the following statement: If
$x=(x_1,\ldots,x_n)\in \bar{K}$, $y=(y_1,\ldots,y_n)\in \bar{K}$ and $\lambda\in(0,1)$, then
\begin{align}\label{winbarK}
w=(|(1-\lambda)x_1+\lambda y_1|,\ldots,|(1-\lambda)x_n+\lambda y_n|)\in \Phi(K)\cap\R^n_+.
\end{align}
If $z\in \partial'\Phi(K)\cap\R^n_+$, then for
$\tilde{x}=(|x_1|,\ldots,|x_n|)\in \Phi(K)\cap\R^n_+$ and $\tilde{y}=(|y_1|,\ldots,|y_n|)\in \Phi(K)\cap\R^n_+$, 
we deduce from $\nu_{\Phi(K)}(z)\in\R^n_+$ (see \eqref{den exw idea pws na to pw}) and
$|(1-\lambda)x_i+\lambda y_i|\leq (1-\lambda)|x_i|+\lambda |y_i|$ that
\begin{align*}
\langle w,\nu_{\Phi(K)}(z)\rangle&\leq 
\langle (1-\lambda)\tilde{x}+\lambda \tilde{y},\nu_{\Phi(K)}(z)\rangle=
 (1-\lambda)\langle\tilde{x},\nu_{\Phi(K)}(z)\rangle
+\lambda\langle \tilde{y},\nu_{\Phi(K)}(z)\rangle\\
&\leq (1-\lambda)\langle z,\nu_{\Phi(K)}(z)\rangle
+\lambda\langle z,\nu_{\Phi(K)}(z)\rangle=\langle z,\nu_{\Phi(K)}(z)\rangle.
\end{align*}
Since $\Phi(K)$ is the intersection of the supporting halfspaces at the smooth boundary points (see
Theorem~2.2.6 in Schneider \cite{book4}), we conclude \eqref{winbarK}, and in turn 
Lemma~\ref{unco.convex}.
\end{proof}

Now we are ready to give the proof of Proposition \ref{basikothe}. 

\begin{proof}[Proof of Proposition \ref{basikothe}]
Let $\bar{G}$ be the
group generated by 	the orthogonal reflections  through the hyperplanes
$H_1,\ldots,H_n$ of $\R^n$ with $H_1\cap\ldots\cap H_n=\{o\}$. Then 
 the corresponding irreducible subspaces $E_1,\ldots,E_k$ coincide for $\bar{G}$ and for its closure $G$.
Let $\bar{G}_\alpha,G_\alpha\subset \mathcal{O}(E_\alpha)$, $\alpha=1,\cdots,n$, be the subgroups
such that $\bar{G}=\bar{G}_1\times\ldots\times\bar{G}_k$ and
$G=G_1\times\ldots\times G_k$ where $G_\alpha$ is the closure of $\bar{G}_\alpha$ in  
$\mathcal{O}(E_\alpha)$ for $\alpha=1,\cdots,n$. In particular, $\bar{G}_\alpha$ is generated by reflections
through all $H_i\cap E_\alpha$ such that $E_\alpha\not\subset H_i$ where writing $H_i=w_i^\bot$ for $w_i\in S^{n-1}$
and $i=1,\ldots,n$, we have $E_\alpha={\rm lin}\{w_i:\,E_\alpha\not\subset H_i\}$.

If $\bar{G}_\alpha$ is finite, then we simply define $\widetilde{G}_\alpha=\bar{G}_\alpha=G_\alpha$.
If $\bar{G}_\alpha$ is infinite, then $\bar{G}_\alpha=\mathcal{O}(E_\alpha)$ according to
Lemma~\ref{infinite-group}; therefore, we may choose $\widetilde{G}_\alpha$ to be the symmetry group of a regular simplex of $E_\alpha$ centered at the origin. In particular, for each $\alpha=1,\ldots,k$,
$\widetilde{G}_\alpha$ is finite and acts irreducibly on $E_\alpha$, and
  let $C_\alpha\subset E_\alpha$ be a Weyl chamber for the action of $\widetilde{G}_\alpha$  as in Lemma~\ref{prop.for.chambers}.

We define $\widetilde{G}=\widetilde{G}_1\times\ldots\times\widetilde{G}_k\subset O(n)$ and
$C=\oplus_{\alpha=1}^k C_\alpha$. We deduce Proposition \ref{basikothe} (ii) and (iii) from
 Lemma~\ref{Propert.H,C} (ii) and (iii).

For Proposition \ref{basikothe} (iv), the walls of $C$ are of the form $W+E_\alpha^\bot$
for $\alpha=1,\ldots,k$ and wall $w$ of $C_\alpha$ in $E_\alpha$. For $1\leq i<j\leq n$,
$\langle v_i,v_j\rangle\leq 0$ follows from  Lemma~\ref{prop.for.chambers} (iv)
 if $v_i,v_j\in E_\alpha$,
and from the othogonality of $E_\alpha$ and $E_\beta$
if $v_i \in E_\alpha$ and $v_j\in E_\beta$ for $\alpha\neq\beta$.

For Proposition \ref{basikothe} (v), we deduce from Lemma~\ref{Propert.H,C} that
$\langle x,y\rangle\geq 0$ holds for $x,y\in C$. Combining this with
Proposition \ref{basikothe} (ii), Lemma~\ref{autokainden} (i) and Lemma~\ref{unco.convex}
yields Proposition \ref{basikothe} (v). 
\end{proof}

\section{Log-Brunn-Minkowski inequality for convex bodies with symmetries}
\label{secLBM}

First we consider the version of Theorem~\ref{logBMsymmetry} where each linear reflection is an orthogonal reflection.

\begin{theo}
\label{logBMwithsymmetries}
	Let  $\lambda\in(0,1)$. If $A_1,\ldots,A_n$ are orthogonal reflections through the
hyperplanes $H_1,\ldots,H_n$ such that $H_1\cap\ldots\cap H_n=\{o\}$, and the convex bodies $K$ and $L$ are invariant under 
	$A_1,\ldots,A_n$, then
	\begin{align}\label{log.B.M.inequality}
	V((1-\lambda)\cdot K +_0 \lambda\cdot L)\geq V(K)^{1-\lambda}V(L)^{\lambda}.
	\end{align}
	In addition, equality holds if and only if 
 $K=K_1\oplus\cdots\oplus K_m$ and $L=L_1\oplus\cdots\oplus L_m$ for 
$m\geq 1$ and compact convex sets $K_i,L_i$  invariant under $A_i$, $i=1,\cdots,m$, 
having dimension at least one and satisfying $K_i=c_iL_i$ for $c_i>0$ for 
$i=1,\cdots,m$, and  $\sum_{i=1}^m{\rm dim}K_i=n$.
\end{theo}
\begin{proof}
Let $G\subset O(n)$ be the closure of the group generated by $A_1,\ldots,A_n$.
We use the notation of Proposition~\ref{basikothe} applied to this $G$. 
In particular, for some $k\geq 1$,  $\R^n=\oplus_{\alpha=1}^kE_{\alpha}$
for non-trivial linear subspaces $E_1,\ldots,E_k$ 
where $E_1,\ldots,E_k$ are pairwise orthogonal if $k\geq 2$.
In addition, $C=\oplus_{\alpha=1}^k C_\alpha$ is the simplicial cone of Proposition~\ref{basikothe} where $C_\alpha$ is the Weyl chamber for the finite group $\widetilde{G}_\alpha\subset \mathcal{O}(E_\alpha)$ generated by reflections through the walls of $C_\alpha$ in $E_\alpha$ and acting irreducibly on $E_\alpha$  for
$\alpha=1,\ldots,k$, and  
$\widetilde{G}=\widetilde{G}_1\times\ldots\times\widetilde{G}_k$ is a subgroup of $G$.

We fix an orthonormal basis $e_1,\ldots,e_n$ of $\R^n$ such that
 $\{e_i:\,e_i\in E_\alpha\}$ spans $E_\alpha$ for $\alpha=1,\ldots,k$, 
and hence $\R^n_+={\rm pos}\{e_1,\ldots,e_n\}$,
and let $\Phi\in{\rm GL}(n)$ be a linear transform such that $\Phi C_\alpha=\R^n_+\cap E_\alpha$ for $\alpha=1,\ldots,k$. We deduce that 
\begin{enumerate}[(a)]
\item $\Phi(C)=\R^n_+$;
\item $\Phi E_\alpha=E_\alpha$ for $\alpha=1,\ldots,k$.
\end{enumerate} 

Since $K$ and $L$ are convex bodies invariant under $G$, their $L_0$ sum $(1-\lambda)\cdot K +_0 \lambda\cdot L$
is also invariant under $G$, and in turn  invariant under $\widetilde{G}$.
We write $|\widetilde{G}|$ to denote the cardinatility of $\widetilde{G}$.
It follows from  the linear covariance of the logarithmic sum and Proposition~\ref{basikothe} (iii) that
	\begin{align}
\nonumber
	V((1-\lambda)\cdot K +_0 \lambda\cdot L)&=\sum_{g\in \widetilde{G}}
V(g C\cap[ (1-\lambda)\cdot K +_0 \lambda\cdot L])\\
\label{csdvdbvsdcs}
&=|\widetilde{G}|\cdot V(C\cap[(1-\lambda)\cdot K +_0 \lambda\cdot L]).	
\end{align}

Since any convex body is the intersection of the supporting halfspaces at the smooth boundary points according to
Theorem~2.2.6 in Schneider \cite{book4}, we deduce from Lemma~\ref{autokainden} (ii) and \eqref{L0sumsmooth} that
$$
C\cap[(1-\lambda)\cdot K +_0 \lambda\cdot L]=
\left\{x\in C:\langle x,u\rangle \leq h_{K}(u)^{1-\lambda}h_{L}(u)^{\lambda}\mbox{ }\forall\,u\in C\right\}.
$$

	Let $\bar{K}$ and $\bar{L}$ the unconditional sets defined by $\bar{K}\cap\R^n_+=\Phi(K\cap C)$ and $\bar{L}\cap\R^n_+=\Phi(L\cap C)$ respectively. Proposition~\ref{basikothe} \eqref{the3} implies that $\bar{K}$ and $\bar{L}$ are unconditional convex bodies. We observe that if $u\in\R^n_+$, then
\begin{equation}
\label{barKPhiK}
h_{\bar{K}}(u)=\max_{x\in\bar{K}\cap \R^n_+}\langle u,x\rangle\leq h_{\Phi K}(u)\mbox{ \ and \ }
h_{\bar{L}}(u)=\max_{x\in\bar{L}\cap \R^n_+}\langle u,x\rangle\leq h_{\Phi L}(u).
\end{equation}
	
The key observation is that
\begin{align}\label{outeafto}
	V(\R^n_+\cap[(1-\lambda)\cdot \Phi(K) +_0 \lambda\cdot \Phi(L)])
\geq V(\R^n_+\cap[(1-\lambda)\cdot\overline{K}+_0\lambda\cdot\overline{L}]),
\end{align}
which follows from (a), \eqref{barKPhiK} and $\Phi^{-t}C\subset \R^n_+$ (see Lemma~\ref{unco.convex}) and 
\begin{eqnarray*}
	\R^n_+\cap \Phi[(1-\lambda)\cdot K +_0 \lambda\cdot L]&=&
\Phi\left(\left\{x\in C:\langle x,u\rangle \leq h_{K}(u)^{1-\lambda}h_{L}(u)^{\lambda}
\mbox{ }\forall\,u\in C\right\}
\right)\\
&=&
\left\{x\in \R^n_+:\langle x,v\rangle \leq 
h_{\Phi K}(v)^{1-\lambda}h_{\Phi L}(v)^{\lambda}\mbox{ }\forall\,v\in  \Phi^{-t}C\right\}\\
&\supset&
\left\{x\in \R^n_+:\langle x,v\rangle \leq 
h_{\bar{K}}(v)^{1-\lambda}h_{\bar{L}}(v)^{\lambda}\mbox{ }\forall\,v\in  \Phi^{-t}C\right\}\\
&\supset&\left\{x\in \R^n_+:\langle x,v\rangle \leq
 h_{\bar{K}}(v)^{1-\lambda}h_{\bar{L}}(v)^{\lambda}\mbox{ }\forall\,v\in  \R^n_+\right\}\\
&=& \R^n_+\cap[(1-\lambda)\cdot\overline{K}+_0\lambda\cdot\overline{L}].
\end{eqnarray*}

	From \eqref{csdvdbvsdcs}, \eqref{outeafto} and Logarithmic Brunn Minkowski inequality
Theorem~\ref{B.F.M.S.} for unconditional convex bodies, we deduce
	\begin{eqnarray}
	\nonumber V((1-\lambda)\cdot K +_0 \lambda\cdot L)&=& |\widetilde{G}|\cdot	V(C\cap[(1-\lambda)\cdot K +_0 \lambda\cdot L])\\ \nonumber
	&=&\frac{|\widetilde{G}|}{|{\rm det}\Phi|}\cdot	
	V(\R^n_+\cap[(1-\lambda)\cdot \Phi(K) +_0 \lambda\cdot \Phi(L)])\\ \nonumber
	&\geq&\frac{|\widetilde{G}|}{|{\rm det}\Phi|}\cdot		
V(\R^n_+\cap[(1-\lambda)\cdot \bar{K} +_0 \lambda\cdot \bar{L}])\\ \nonumber
	&=&\frac{|\widetilde{G}|}{2^n|{\rm det}\Phi|}	\cdot	
V((1-\lambda)\cdot \bar{K} +_0 \lambda\cdot \bar{L})\\ \label{equali-BAR-KL}
	&\geq&\frac{|\widetilde{G}|}{2^n|{\rm det}\Phi|}\cdot		
V(\bar{K})^{1-\lambda}V(\bar{L})^{\lambda} \\ \nonumber
	&=&	V(K)^{1-\lambda}V(L)^{\lambda},
	\end{eqnarray}	
proving the Logarithmic Brunn-Minkowski inequality \eqref{log.B.M.inequality}.\\

	Assume now that we have equality in \eqref{log.B.M.inequality}. In particular,  equality holds for the unconditional convex bodies $\bar{K}$ and $\bar{L}$ in \eqref{equali-BAR-KL}. Therefore, Theorem~\ref{B.F.M.S.} implies that
$\bar{K}=\bar{K}_1\oplus\cdots\oplus\bar{K}_m$ and 
$\bar{L}=\bar{L}_1\oplus\cdots\oplus\bar{L}_m$ for some $m\geq 1$, where $\bar{K}_{\beta}$ and $\bar{L}_{\beta}$ are unconditional convex sets, $\bar{K}_{\beta}=\theta_\beta\bar{L}_{\beta}$ for some 
$\theta_{\beta}>0$, $\beta=1,\cdots,m$, and
$\sum_{\beta=1}^m{\rm dim}\bar{K}_\beta=n$.

If $m=1$, then
\begin{eqnarray*}
	K&=&\bigcup_{g\in \widetilde{G}}g(K\cap C)
	=\bigcup_{g\in \widetilde{G}}g \circ\Phi^{-1}(\bar{K}\cap \R^n_+)
	=\bigcup_{g\in \widetilde{G}}g \circ\Phi^{-1}(\theta_1\bar{L}\cap \R^n_+)\\
&=&\theta_1\bigcup_{g\in \widetilde{G}}g(L\cap C)
	=\theta_1\cdot L;
	\end{eqnarray*}
 therefore,  $K$ and $L$ are dilates.
	
If  $m\geq 2$, then we write $\bar{E}_{\beta}={\rm lin}\bar{K}_{\beta}$ for $\beta=1,\cdots,m$,
 and hence  $\R^n=\oplus_{\beta=1}^m\bar{E}_{\beta}$.

We claim that each $\bar{E}_{\beta}$ is the direct sum of some $E_{\alpha}$; namely, 
there exists some non-empty $\Xi_\beta\subset\{1,\cdots,k\}$ such that
	\begin{align}\label{fragkosiriani}
	\bar{E}_{\beta}=\oplus_{\alpha\in \Xi_{\beta}}E_{\alpha}
	\end{align}
We suppose that \eqref{fragkosiriani} does not hold, and seek a contradiction.
We set $u_i=\Phi^{-1}(e_i)$ and  $v_i=\Phi^{t}(e_i)$ for $i=1,\ldots,n$; therefore,
$C={\rm pos}\{u_1,\ldots,u_n\}$ and  $C^*={\rm pos}\{v_1,\ldots,v_n\}$.                           
For any $\alpha=1,\ldots,k$ and $\beta=1,\ldots,m$, we consider
$$
I_\alpha=\{i:\,e_i\in E_\alpha\}
\mbox{ \ and \ }
J_\beta=\{j:\,e_j\in \bar{E}_\beta\}.
$$
Since $\{1,\ldots,n\}$ is partitioned in two ways  once into $I_1,\ldots,I_k$,
and secondly into $J_1,\ldots,J_m$, the indirect hypothesis
yields  there exist $\tilde{\alpha}\in\{1,\cdots,k\}$ and $\tilde{\beta}\in\{1,\cdots,m\}$ such that
$I_{\tilde{\alpha}}\cap J_{\tilde{\beta}}$ is non-empty and is a proper subset of
$I_{\tilde{\alpha}}$. It follows from Lemma~\ref{prop.for.chambers} (v) applied to $C_{\tilde{\alpha}}$
and the partition $I_{\tilde{\alpha}}=(I_{\tilde{\alpha}}\cap J_{\tilde{\beta}})\cup(I_{\tilde{\alpha}}\backslash J_{\tilde{\beta}})$ that there exist 
\begin{equation}
\label{vpvq}
\mbox{$p\in I_{\tilde{\alpha}}\cap J_{\tilde{\beta}}$ and
$q\in I_{\tilde{\alpha}}\backslash J_{\tilde{\beta}}$
such that $\langle v_p,v_q\rangle<0$.}
\end{equation}

Since for any convex body, smooth boundary points are dense on the boundary, there exists
a $z_0\in{\rm relint}\,(\bar{K}_{\tilde{\beta}}\cap \R^n_+)$ and $s>0$ such that 
$z=z_0+se_p\in \partial' \bar{K}_{\tilde{\beta}}\cap \R^n_+$. It follows that
$\langle \nu_{\bar{K}_{\tilde{\beta}}\cap \R^n_+}(z),e_p\rangle>0$, and hence
\begin{equation}
\label{nuKtildebetaz}
\nu_{\bar{K}_{\tilde{\beta}}\cap \R^n_+}(z)=\sum_{j\in J_{\tilde{\beta}}}\gamma_je_j
\mbox{ \ where $\gamma_j\geq 0$ for $j\in J_{\tilde{\beta}}$ and
$\gamma_p>0$.}
\end{equation}
We choose a $y\in{\rm relint}\sum_{\beta\neq \tilde{\beta}}(\bar{K}_\beta\cap \R^n_+)$; therefore,
$z+y\in\partial'\bar{K}\cap \R^n_+$ and we deduce from \eqref{nuKtildebetaz} that
\begin{equation}
\label{nubarKyz}
\nu_{\bar{K}}(z+y)=\nu_{\bar{K}_{\tilde{\beta}}\cap \R^n_+}(z)=\sum_{j\in J_{\tilde{\beta}}}\gamma_je_j
\mbox{ \ where $\gamma_j\geq 0$ for $j\in J_{\tilde{\beta}}$ and
$\gamma_p>0$.}
\end{equation}
Writing $\Phi^{-1}z=z'$ and $\Phi^{-1}y=y'$, it follows that
\begin{eqnarray}
\label{yzprime}
z'+y'&\in&\partial'K\cap C\\
\nonumber
\nu_{K}(z'+y')&=&\theta\,\Phi^{t}\nu_{\bar{K}}(z+y)
\mbox{ \ for $\theta=\|\Phi^{t}\nu_{\bar{K}}(z+y)\|^{-1}$},
\end{eqnarray}
which combined with \eqref{nubarKyz} leads to
 \begin{equation}
\label{nuKyz}
\nu_{K}(z'+y')=\sum_{j\in J_{\tilde{\beta}}}\theta\gamma_jv_j
\mbox{ \ where $\gamma_j\geq 0$ for $j\in J_{\tilde{\beta}}$ and
$\gamma_p>0$.}
\end{equation}
In turn, we deduce from \eqref{nuKyz}, $\langle v_p,v_q\rangle< 0$ (see \eqref{vpvq}) and
 $\langle v_p,v_j\rangle\leq 0$ for $j\in J_{\tilde{\beta}}$ 
(see Proposition~\ref{basikothe} (iv)) that
$$
\langle v_p,\nu_{K}(z'+y')\rangle=\theta\gamma_q\langle v_p,v_q\rangle
+\sum_{j\in J_{\tilde{\beta}}\atop j\neq q}\theta\gamma_j
\langle v_p,v_j\rangle<0,
$$
and hence $C=\{x:\,\langle x,v_i\rangle\geq 0\mbox{ for }i=1,\ldots,n\}$
(see \eqref{simplisial convex cone form2}) yields $\nu_{K}(z'+y')\not \in C$.

On the other hand, combining \eqref{yzprime} and Proposition~\ref{basikothe} (v)
implies that $\nu_{K}(z'+y') \in C$. This contradiction finally proves \eqref{fragkosiriani}.

We recall that $M|E$ denotes the orthogonal projection of a compact convex set $M$ onto a linear subspace $E$.
We deduce from \eqref{fragkosiriani} and Proposition~\ref{basikothe} (i), (ii) and (iii) that
for each $\beta=1,\ldots,m$, there exist convex compact sets $K_\beta,L_\beta\subset \bar{E}_\beta$ such that
\begin{eqnarray}
\nonumber
K_\beta\cap C&=&\Phi^{-1}(\bar{K}_\beta\cap \R^n_+);\\
\label{KKbeta}
K&=&K_\beta+(K|\bar{E}_\beta^\bot)\mbox{ \ and \ }K|\bar{E}_\beta=K_\beta;\\
\nonumber
L_\beta\cap C&=&\Phi^{-1}(\bar{L}_\beta\cap \R^n_+);\\
\label{LLbeta}
L&=&L_\beta+(L|\bar{E}_\beta^\bot)\mbox{ \ and \ }L|\bar{E}_\beta=L_\beta.
\end{eqnarray}
In turn, we verify that \eqref{KKbeta} and \eqref{LLbeta} yield that
\begin{equation}
\label{KLbetaclaim}
K=\oplus_{\beta=1}^mK_\beta\mbox{ \ and \ }L=\oplus_{\beta=1}^mL_\beta
\end{equation}
by induction on $m\geq 2$. We only provide the argument in the case of $K$, because the argument for $L$ is similar.

If $m=2$, then $\bar{E}_1^\bot=\bar{E}_2$, and so \eqref{KLbetaclaim} readily follows.

If $m\geq 3$, then let $K'=K|\bar{E}_m^\bot$. Let $1\leq \beta\leq m-1$.
The main observation we use is that if $\Pi_0\subset\Pi$ are linear subspaces, then
$(X|\Pi)|\Pi_0=X|\Pi_0$ for $X\subset \R^n$.
  On the one hand, we deduce from $\bar{E}_\beta\subset  \bar{E}_m^\bot$
that
$$
K'|\bar{E}_\beta=(K|\bar{E}_m^\bot)|\bar{E}_\beta=K|\bar{E}_\beta=K_\beta.
$$
On the other hand, we also use that if $X\subset \bar{E}_\beta^\bot$, then  
 $X|\bar{E}_m^\bot=X|(\bar{E}_\beta^\bot\cap \bar{E}_m^\bot)$ follows from $\bar{E}_\beta\subset  \bar{E}_m^\bot$. Therefore,
\begin{eqnarray*}
K'&=&\Big(K_\beta+(K|\bar{E}_\beta^\bot)\Big)|\bar{E}_m^\bot=
K_\beta|\bar{E}_m^\bot+(K|\bar{E}_\beta^\bot)|\bar{E}_m^\bot=
K_\beta+(K|\bar{E}_\beta^\bot)|(\bar{E}_m^\bot\cap \bar{E}_\beta^\bot)\\
&=&K_\beta+K|(\bar{E}_m^\bot\cap \bar{E}_\beta^\bot)
=K_\beta+K'|(\bar{E}_m^\bot\cap \bar{E}_\beta^\bot),
\end{eqnarray*}
implying $K'=\oplus_{\beta=1}^{m-1}K_\beta$ by induction on $m$. Since $K=K_m+K'$ by \eqref{KKbeta},
we conclude \eqref{KLbetaclaim}.

As $K$, $L$ and $\bar{E}_\beta$ are invariant under $G$, also $K_\beta$ and $L_\beta$ are invariant under $G$
for $\beta=1,\ldots,m$.
Since $\bar{K}_{\beta}=\theta_\beta\bar{L}_{\beta}$, we also deduce that $K_{\beta}=\theta_\beta L_{\beta}$
for $\beta=1,\ldots,m$, verifying the necessity of the condition in Theorem~\ref{logBMwithsymmetries}
in the case of equality in \eqref{log.B.M.inequality}.

Finally, if $K$ and $L$ are convex bodies with $o\in{\rm int}K$ and $o\in{\rm int}L$,
and $K=K_1+\cdots+ K_m$ and $L=L_1+\cdots + L_m$ for 
$m\geq 1$ and compact convex sets $K_i,L_i$, $i=1,\cdots,m$, 
having dimension at least one and satisfying $o\in K_i$ and $K_i=\theta_iL_i$ for $\theta_i>0$ for 
$i=1,\cdots,m$, and  $\sum_{i=1}^m{\rm dim}K_i=n$, 
then equality holds in \eqref{log.B.M.inequality} even without symmetry assumption 
according to Lemma~\ref{sumlowerlogBMequa}.
This completes the proof of Theorem~\ref{logBMwithsymmetries}. 
\end{proof}

We are ready to prove Theorem~\ref{logBMsymmetry}. 

\begin{proof}[Proof of Theorem~\ref{logBMsymmetry}]

According to John's theorem (see Schneider \cite{book4}), there exists a unique ellipsoid $E$ of minimal volume containing $K$, which is also known as L\"owner ellipsoid. It follows that $E$ is also invariant under $A_1,\ldots,A_n$. For a linear transform 
 $\Phi\in{\rm GL}(n)$ satisfying that $\Phi E=B^n$, the linear transforms $A'_i=\Phi A_i\Phi^{-1}$, $i=1,\ldots,n$ leave $B^n$ invariant, thus $A'_i$ is an orthogonal reflection through the hyperlane $H'_i=\Phi H_i$
where $H_1\cap\ldots\cap H_n=\{o\}$. In addition, the convex bodies $K'=\Phi K$ and $L'=\Phi L$ are invariant under
$A'_1,\ldots,A'_n$.

Finally, applying Theorem~\ref{logBMwithsymmetries} to $K'$ and $L'$, and using the linear covariance of the $L_0$-sum
(see \eqref{linearinv}), we conclude Theorem~\ref{logBMsymmetry}.  
\end{proof}

\section{From log-B.M. for Lebesgue measure to log-B.M. with $e^{-\phi(x)}dx$ where $\phi(x)$ is any rotationally invariant convex function}
\label{secGaussian}

A function $f:\,\Omega \to [0,\infty)$ on a convex subset $\Omega\subset \R^n$ is called log-concave if
$f((1-s)x+sy)\geq f(x)^{1-s}f(y)^s$ for any $s\in[0,1]$ and $x,y\in\Omega$; namely, if
$f=e^{-\phi(x)}$ for a convex function $\phi:\R^n\to \R\cup \infty$. Analogously,
a measure $\nu$ on $\R^n$ is called log-concave if 
$d\nu(x)=f(x)\,dx$ for a log-concave function $f$ on $\R^n$ (which property is equivalent saying that
$\nu((1-\lambda)A+\lambda\,B)\geq \nu(A)^{1-\lambda}\nu(B)^\lambda$ for any $A,B\subset \R^n$ compact
according to Borell \cite{Bor75b}). 

We note that Saroglou \cite{Sar16} proved that on the class of $o$-symmetric convex bodies and measures,
the logarithmic-Brunn-Minkowski inequality  for the Lebesgue measure implies the logarithmic-Brunn-Minkowski inequality
for any log-concave measure.
  In other words, according to
Theorem~3.1 in \cite{Sar16}, if \eqref{logBMeq} holds for any $o$-symmetric convex bodies $K,L$ in $\R^n$, then for
any even convex $\varphi:\,\R^n\to(-\infty,\infty]$ function, we have
\begin{align*}
\int_{(1-\lambda)\cdot K +_0 \lambda\cdot L}e^{-\varphi(x)}dx \geq \Big( \int_K e^{-\varphi(x)}dx \Big)^{1-\lambda} 
\Big( \int_Le^{-\varphi(x)}dx \Big)^{\lambda}
\end{align*} 
 for any $o$-symmetric convex bodies  $K,L$. 

However, the proof of Theorem~3.1 in \cite{Sar16} does not actually use $o$-symmetry but a somewhat weaker
property of log-concave measures with rotational symmetry. Let $\varphi(x)=\psi(\|x\|)$ for an increasing convex function
$\psi:[0,\infty)\to(-\infty,\infty]$, and let $G$ be the subgroup generated by orthogonal reflections through
the linear hyperplanes $H_1,\ldots,H_n$ with $H_1\cap\ldots \cap H_n=\{o\}$, and hence if
$M$ is a convex body invariant under $G$, then
$M\cap\{\varphi\leq r\}$ is also invariant under $G$ for any $r>\varphi(o)=\psi(0)$.
 It follows that
if $K$ and $L$ are convex bodies invariant under $G$, then 
 all bodies
used in the proofs of Lemma~3.7  and Theorem~3.1 in \cite{Sar16} are also invariant under $G$; therefore,
the argument by Saroglou \cite{Sar16} yields the following theorem implying Theorem~\ref{Gaussian-BMsymmetry} in the case 
$\psi(t)=t^2$.

\begin{theo}
	\label{Logconv-BMsymmetry}
	Let  $\lambda\in(0,1)$, let $\varphi(x)=\psi(\|x\|)$ for an increasing convex function
$\psi:[0,\infty)\to(-\infty,\infty]$ and let $d\nu(x)=e^{-\varphi(x)}\,dx$ be the corresponding log-concave measure on $\R^n$. 
If  $H_1\cap\ldots\cap H_n=\{o\}$ holds for the linear hyperplanes $H_1,\ldots,H_n$, and the convex bodies $K$ and $L$ are invariant under the orthogonal reflections through $H_1,\ldots,H_n$, then
	$$
	\nu((1-\lambda)\cdot K +_0 \lambda\cdot L)\geq \nu(K)^{1-\lambda} \nu(L)^\lambda.
	$$
\end{theo}

\section{The proof of Theorem~\ref{logMsymmetry} and Theorem~\ref{Minkowski-symmetry-uniqueness}}
\label{secuniqueness}

The following Proposition~\ref{logMcalC} was stated in the case when $\mathcal{C}$ is
the family of orgin symmetric bodies in B\"or\"oczky, Lutwak, Yang, Zhang \cite{BLYZ12}, but the method
of  \cite{BLYZ12} (taking the derivative of \eqref{logMcalCcond} with respect to $\lambda$ at $\lambda=0$) actually yields the following slightly more general statement.

\begin{prop}
\label{logMcalC}
Let $\mathcal{C}$ be a class of convex bodies containing the origin in their interior
such that $\mathcal{C}$ is
 closed under dilation and the $L_0$-sum (i.e. $(1-\lambda)\cdot K+_0\lambda\cdot L\in\mathcal{C}$ for any $K,L\in\mathcal{C}$, $\lambda\in[0,1]$), and
\begin{equation}
\label{logMcalCcond}
V((1-\lambda)\cdot K +_0 \lambda\cdot L)\geq V(K)^{1-\lambda} V(L)^\lambda
\end{equation}
holds for any $K,L\in\mathcal{C}$. Then  
$$
\int_{S^{n-1}}\log \frac{h_L}{h_K}\,dV_K\geq \frac{V(K)}n\log\frac{V(L)}{V(K)}
$$
for any $K,L\in\mathcal{C}$ with equality if and only if 
$V(\frac12\cdot K +_0 \frac12\cdot L)= V(K)^{1/2} V(L)^{1/2}$.
\end{prop}

Now we recall the argument in B\"or\"oczky, Lutwak, Yang, Zhang \cite{BLYZ12} to prove 
Theorem~\ref{Minkowski-symmetry-uniqueness} in the slightly more general form of 
Proposition~\ref{Minkowski-symmetry-uniqueness0}.

\begin{prop}
\label{Minkowski-symmetry-uniqueness0}
Let $\mathcal{C}$ be a class of convex bodies containing the origin in their interior
such that $\mathcal{C}$ is
 closed under dilation and the $L_0$-sum, and
\begin{equation}
\label{logMcalCcond1}
V((1-\lambda)\cdot K +_0 \lambda\cdot L)\geq V(K)^{1-\lambda} V(L)^\lambda
\end{equation}
holds for any $K,L\in\mathcal{C}$.
If $V_K=V_L$ for $K,L\in\mathcal{C}$, then
$V(\frac12\cdot K +_0 \frac12\cdot L)=V(K)^{1/2}V(L)^{1/2}$.
\end{prop}
\begin{proof}
We deduce from $V_K=V_L$ and the log-Minkowski inequality Theorem~\ref{logMcalC}
that
	\begin{align*}
	\int_{S^{n-1}}{\rm log}h_LdV_L&=\int_{S^{n-1}}{\rm log}h_LdV_K\geq 
\int_{S^{n-1}}{\rm log}h_KdV_K=\int_{S^{n-1}}{\rm log}h_KdV_L\\
	&\geq \int_{S^{n-1}}{\rm log}h_LdV_L.
	\end{align*}
Thus we have equality in Theorem~\ref{logMcalC}, proving
$V(\frac12\cdot K +_0 \frac12\cdot L)=V(K)^{1/2}V(L)^{1/2}$.
\end{proof}

We observe that Theorem~\ref{logMsymmetry} follows from Theorem~\ref{logBMsymmetry}
and Proposition~\ref{logMcalC}, and 
 Theorem~\ref{Minkowski-symmetry-uniqueness} follows from
Theorem~\ref{logBMsymmetry} and  Proposition~\ref{Minkowski-symmetry-uniqueness0}.

\section{Appendix - Equality case in the log Brunn-Minkowski inequality for unconditional convex bodies}
\label{secAppendix}

This section is dedicated to the proof of the equality case of Theorem \ref{B.F.M.S.} from \cite{Sar15}, and correct a slight mistake in \cite{Sar15}. In particular, we show the case if $K$ and $L$ are unconditional convex bodies and
\begin{equation}
\label{Sarogloucorrected}
V((1-\lambda)\cdot K+_0\lambda\cdot L)=V(K)^{1-\lambda}V(L)^{\lambda},
\end{equation}
then $K$ and $L$ have dilated vector summands as described in Theorem \ref{B.F.M.S.}. For the convenience of the reader,
we review the whole argument.

The classical coordinatewise product of two unconditional convex bodies $K$ and $L$ in $\R^n$ is
$$
K^{1-\lambda}\cdot L^\lambda=
\{(\pm |x_1|^{1-\lambda}|y_1|^\lambda,\ldots,\pm |x_n|^{1-\lambda}|y_n|^\lambda):\,
(x_1,\ldots,x_n)\in K\mbox{ and }(y_1,\ldots,y_n)\in L \}.
$$
We recall the following well-known facts where (i) is due to 
Bollobas, Leader \cite{BoL95}, and (ii) is due to Saroglou \cite{Sar15}.

\begin{lemma}[Bollobas-Leader, Saroglou]
\label{coordinatewise-convex}
If $K$ and $L$ are unconditional convex bodies in $\R^n$ with respect to the same orthonormal basis and $\lambda\in(0,1)$, then
\begin{description}
\item{(i)} $K^{1-\lambda}\cdot L^\lambda$ is an unconditional convex body;
\item{(ii)} $K^{1-\lambda}\cdot L^\lambda\subset (1-\lambda)\cdot K+_0 \lambda \cdot L$.
\end{description}
\end{lemma}
\begin{proof}
For (i), we observe that if $t_i\geq |z_i|$ for $i=1,\ldots,n$ and $(t_1,\ldots,t_n)\in K^{1-\lambda}\cdot L^\lambda$, then
\begin{equation}
\label{corner}
(z_1,\ldots,z_n)\in K^{1-\lambda}\cdot L^\lambda.
\end{equation}
Now if $s\in(0,1)$, $p=(p_1,\ldots,p_n)\in K^{1-\lambda}\cdot L^\lambda$ and 
$\tilde{p}=(\tilde{p}_1,\ldots,\tilde{p}_n)\in K^{1-\lambda}\cdot L^\lambda$, then
there exist $x=(x_1,\ldots,x_n)\in K\cap \R^n_+$, $\tilde{x}=(\tilde{x}_1,\ldots,\tilde{x}_n)\in K\cap \R^n_+$,
$y=(y_1,\ldots,y_n)\in L\cap \R^n_+$ and $\tilde{y}=(\tilde{y}_1,\ldots,\tilde{y}_n)\in K\cap \R^n_+$
such that $|p_i|=x_i^{1-\lambda}y_i^\lambda$ and $|\tilde{p}_i|=\tilde{x}_i^{1-\lambda}\tilde{y}_i^\lambda$
for $i=1,\ldots,n$. Therefore first applying the the triangle and secondly the H\"older inequality yield that if $i=1,\ldots,n$, then
\begin{eqnarray}
\nonumber
|(1-s)p_i+s\tilde{p}_i|&\leq &(1-s)|p_i|+s\,|\tilde{p}_i|
=(1-s)x_i^{1-\lambda}y_i^\lambda+s\,\tilde{x}_i^{1-\lambda}\tilde{y}_i^\lambda\\
\label{coordinatwise-convex-est}
&\leq &((1-s)x_i+s\,\tilde{x}_i)^{1-\lambda}
((1-s)y_i+s\tilde{y}_i)^\lambda.
\end{eqnarray}
Combining \eqref{corner}, \eqref{coordinatwise-convex-est} and the convexity of $K$ and $L$ implies that
$(1-s)p+s\,\tilde{p}\in K^{1-\lambda}\cdot L^\lambda$, verifying (i).

For (ii), if $u=(u_1,\ldots,u_n)\in S^{n-1}$ and $(z_1,\ldots,z_n)\in K^{1-\lambda}\cdot L^\lambda$, then
there exist $x=(x_1,\ldots,x_n)\in K\cap \R_+^n$ and $y=(y_1,\ldots,y_n)\in L\cap \R_+^n$
such that $|z_i|=x_1^{1-\lambda}y_i^\lambda$ for $i=1,\ldots,n$. Therefore for 
$\tilde{u}=(|u_1|,\ldots,|u_n|)\in S^{n-1}$, it follows from the H\"older inequality that
\begin{eqnarray*}
\langle z,u\rangle &\leq & \sum_{i=1}^nx_i^{1-\lambda}y_i^\lambda |u_i|\leq
\left(\sum_{i=1}^nx_i |u_i|\right)^{1-\lambda}\left(\sum_{i=1}^ny_i |u_i|\right)^\lambda\\
&\leq& h_K(\tilde{u})^{1-\lambda}h_L(\tilde{u})^\lambda=h_K(u)^{1-\lambda}h_L(u)^\lambda,
\end{eqnarray*}
proving (ii).
\end{proof}

Next Bollobas, Leader \cite{BoL95} and later indepently Cordero-Erausquin, Fradelizi, Maurey \cite{CEFM04}
proved the "Brunn-Minkowski inequality for the coordinatewise product" \eqref{coordinatewise-ineq}, and the equality case was clarified by Saroglou \cite{Sar15}. The argument is based on the Pr\'ekopa-Leindler inequality (proved in various forms by Pr\'ekopa \cite{Pre71,Pre73}, Leindler \cite{Lei72},
 Borell \cite{Bor75a} and  Brascamp, Lieb \cite{BrL76}) whose equality case was clarified by Dubuc \cite{Dub77}
(see the survey Gardner \cite{gardner}). 
The multiplicative form Theorem~\ref{PL} of the Pr\'ekopa-Leindler inequality was stated first by Ball \cite{Bal} (see also  Uhrin \cite{Uhr94} and Bollobas, Leader \cite{BoL95}).

We note that integration is always with respect to the Lebesgue measure in $\R^n$
in this section. We recall that $f:\,\R^n\to\R_+$ is log-concave if
$f((1-s)x+s\,y)\geq f(x)^{1-s}f(y)^s$ holds for any $s\in(0,1)$ and $x,y\in\R^n$. Here we recall the 
Pr\'ekopa-Leindler inequality for log-concave functions.

\begin{theo}[Pr\'ekopa, Leindler]
\label{PL}
If integrable log-concave $f,g,h:\,\R^n\to\R_+$ and $\lambda\in(0,1)$ satisfy that
$h((1-\lambda)x+\lambda\,y)\geq f(x)^{1-\lambda}g(y)^\lambda$ for any $x,y\in\R^n$, then
$$
\int_{\R^n}h\geq \left(\int_{\R^n}f\right)^{1-\lambda}\left(\int_{\R^n}g\right)^\lambda,
$$
and if equality holds and $\int_{\R^n}f>0$ and $\int_{\R^n}g>0$, then there exist $a>0$ and $b\in\R^n$ such that
$$
g(x)=a\, f(x+b)\mbox{ \ and \ }h(x)=a^\lambda f(x+\lambda\, b).
$$
\end{theo}

Given an orthonormal basis $e_1,\ldots,e_n$ of $\R^n$ and $t_1,\ldots,t_n>0$, we write $\Phi={\rm diag}\,[t_1,\ldots,t_n]$
to denote the positive definite diagonal matrix with $\Phi e_i=t_ie_i$, $i=1,\ldots,n$. 
In addition, for $\eta\in\R$, we set $\Phi^\eta={\rm diag}\,[t_1^\eta,\ldots,t_n^\eta]$.

\begin{theo}[Bollobas-Leader, Saroglou]
\label{coordinatewise-ineq-lemma}
	If $K_1,K_2$  are unconditional convex bodies in $\R^n$ with respect to the same orthonormal basis and $\lambda\in(0,1)$, then
	\begin{equation}
	\label{coordinatewise-ineq}
	V(K_1^{1-\lambda}\cdot K_2^\lambda)\geq V(K_1)^{1-\lambda} V(K_2)^\lambda.
	\end{equation}
	In addition, equality holds  if and only if there exists a positive definite diagonal matrix $\Phi$ such that
$K_2=\Phi K_1$.
\end{theo}
\noindent{\bf Remark } If $K_2=\Phi K_1$ for a positive definite diagonal matrix $\Phi$, then
$K_1^{1-\lambda}\cdot K_2^\lambda=\Phi^\lambda K_1$.
\begin{proof}
Let $K_0=K_1^{1-\lambda}\cdot K_2^\lambda$, and let us consider the log-concave functions
$$
f_i(x_1,\ldots,x_n)=\mathbf{1}_{K_i}(e^{x_1},\ldots,e^{x_n})e^{x_1+\ldots+x_n}
\mbox{ \ \ \ for $i=0,1,2$,}
$$
and the open sets
$$
\Omega_i=\{(x_1,\ldots,x_n)\in\R^n:\, (e^{x_1},\ldots,e^{x_n})\in{\rm int}\,K_i\}
\mbox{ \ \ \ for $i=0,1,2$.}
$$
For $i=0,1,2$, we have
\begin{equation}
\label{cornervolume}
\int_{\R^n}f_i=\int_{\Omega_i}e^{x_1+\ldots+x_n}dx_1\ldots dx_n=V\left(K_i\cap(0,\infty)^n\right)=V(K_i)/2^n,
\end{equation}
and the definition of the coordinatewise product yields
$$
f_0((1-\lambda)x+\lambda\,y)\geq f_1(x)^{1-\lambda}f_2(y)^\lambda
\mbox{ \ \ \ \ for $x,y\in\R^n$.}
$$
Therefore the condition in the Pr\'ekopa-Leindler inequality Theorem~\ref{PL} is satisfied, and we deduce that
$$
\frac{V(K_0)}{2^n}=\int_{\R^n}f_0\geq \left(\int_{\R^n}f_1\right)^{1-\lambda}
\left(\int_{\R^n}f_2\right)^\lambda
=\left(\frac{V(K_1)}{2^n}\right)^{1-\lambda}
\left(\frac{V(K_2)}{2^n}\right)^\lambda,
$$
proving \eqref{coordinatewise-ineq}. 

Let us assume that we have equality in \eqref{coordinatewise-ineq}, and hence also in the corresponding Pr\'ekopa-Leindler inequality. In particular, there exists $a>0$ and $b\in\R^n$ such that
$$
\mathbf{1}_{K_1}(e^{x_1},\ldots,e^{x_n})e^{x_1+\ldots+x_n}=f_1(x)=af_2(x+b)=
a\mathbf{1}_{K_2}(e^{x_1+b_1},\ldots,e^{x_n+b_n})e^{x_1+b_1+\ldots+x_n+b_n}
$$
for Lebesgue almost all $(x_1,\ldots,x_n)\in\R^n$.  Since $f_i$ is continuous on $\Omega_i$ and on 
$\R^n\backslash({\rm cl}\,\Omega_i)$ for $i=1,2$, it follows that 
$$
\mathbf{1}_{K_1}(e^{x_1},\ldots,e^{x_n})=
ae^{b_1+\ldots+b_n}\mathbf{1}_{K_2}(e^{x_1+b_1},\ldots,e^{x_n+b_n})
$$
for each $(x_1,\ldots,x_n)\in\Omega_1$ and $(x_1,\ldots,x_n)\in\R^n\backslash({\rm cl}\,\Omega_1)$; therefore,
$ae^{b_1+\ldots+b_n}=1$ and
${\rm int}\,K_2\cap \R^n_+=\Phi ({\rm int}\,K_1\cap \R^n_+)$ for $\Phi={\rm diag}[e^{-b_1},\ldots,e^{-b_n}]$. In turn, we conclude
$K_2=\Phi K_1$.

On the other hand, if $K_2=\Phi K_1$ for a positive definite diagonal matrix $\Phi$, then readily $K_1^{1-\lambda}\cdot K_2^\lambda=\Phi^\lambda K_1$, and we have equality in \eqref{coordinatewise-ineq}.
\end{proof}

The upcoming Lemma~\ref{logsumPhi} is the only novel contribution of this manuscript about characterizing the equality case \eqref{Sarogloucorrected} of the log-Brunn-Minkowski inequality for unconditional convex bodies.
When proving the analogue of Lemma~\ref{logsumPhi},
Saroglou \cite{Sar15} assumed that if an unconditional convex body $K$ can't be written as the direct sum of at least two lower dimensional 
unconditional compact convex sets, then there exists $x\in\partial'K$ such that each coordinate of $\nu_K(x)$ is positive.
However, this property may not hold. Let $n\geq 3$, and let
$$
K=\left\{(x_1,\ldots,x_n)\in \R^n:\, \sum_{i=1}^{n-1}x_i^2\leq 1\mbox{ and }
\sum_{i=2}^{n}x_i^2\leq 1\right\}.
$$
In this case,
$$
\nu_K( \partial' K)=\left\{(u_1,\ldots,u_n)\in S^{n-1}:\, u_1 u_n=0\right\};
$$
therefore, there exists no $x\in\partial'K$ such that each coordinate of $\nu_K(x)$ is positive on the one hand,
and Lemma~\ref{SKsumofmany} yields that $K$ can't be written as the direct sum of at least two lower dimensional 
unconditional compact convex sets on the other hand.

\begin{lemma}
\label{logsumPhi}
Let $K$ be an unconditional convex body in $\R^n$, 
 and let 
$\Phi$ be a positive definite diagonal matrix with eigenspaces  $\xi_1,\ldots,\xi_m$, $m\geq 2$
where the eigenvalues corresponding to $\xi_i$ and $\xi_j$ are different if $i\neq j$.
If $V((1-\lambda)\cdot K+_0\lambda\cdot (\Phi K))=V(K)^{1-\lambda} V(\Phi K)^\lambda$
for some $\lambda\in(0,1)$, then
$K=K_1\oplus\ldots\oplus K_m$ where $K_i\subset \xi_i$ is an unconditional compact convex set for $i=1,\ldots,m$.
\end{lemma}
\begin{proof}
Since $\Phi$ is a positive definite diagonal matrix, the definition of the coordinatewise product yields
that $K^{1-\lambda}\cdot (\Phi K)^\lambda=\Phi^\lambda\left(K^{1-\lambda}\cdot K^\lambda\right)$.
On the other hand, Lemma~\ref{coordinatewise-convex} (ii) implies that
$$
K\subset K^{1-\lambda}\cdot K^\lambda\subset (1-\lambda)\cdot K +_0 \lambda\cdot K=K.
$$

Since $V(K)^{1-\lambda} V(\Phi K)^\lambda=(\det \Phi^\lambda)V(K)$, and
$$
\Phi^\lambda K=K^{1-\lambda}\cdot (\Phi K)^\lambda\subset (1-\lambda)\cdot K+_0\lambda\cdot (\Phi K),
$$
 we deduce from Theorem~\ref{coordinatewise-ineq-lemma} that
\begin{equation}
\label{Philambda}
\Phi^\lambda K= (1-\lambda)\cdot K+_0\lambda\cdot (\Phi K).
\end{equation}

We prove Lemma~\ref{logsumPhi} by contradiction; therefore, according to Lemma~\ref{SKsumofmany},
we suppose that there exists a $u=(u_1,\ldots,u_n)\in \nu_K( \partial' K)\backslash(\cup_{i=1}^m \xi_i)$.
Since $K$ is unconditional, we may assume that $u\in\R^n_+$. 
Let $e_1,\ldots,e_n$ be the orthonormal basis of $\R^n$. Possibly after reindexing, 
we may also assume that $e_1\in \xi_1$ and $u_1>0$; moreover,
$e_n\in\xi_m$ and $u_n>0$.

Let $\Phi e_i=t_i e_i$ for $t_i>0$, $i=1,\ldots,m$. Since $e_1\in \xi_1$ and $e_n\in\xi_m$, we have
$t_1\neq t_n$; therefore,
\begin{equation}
\label{non-parallel}
\mbox{neither $u$ nor $\Phi^{-1} \,u$ is parallel to $v=\Phi^{-\lambda} u$.}
\end{equation}
As $K$ is unconditional and $u \in \nu_K( \partial' K)$, we have $u=\nu_K(x)$ for some
$x=(x_1,\ldots,x_n)\in\R^n_+\cap \partial' K$. As $x$ is a smooth boundary point and $u_1,u_n>0$, it follows that
$x_1,x_n>0$.

We observe that $v=\Phi^{-\lambda}u$ is an exterior normal at the smooth boundary point 
$\Phi^\lambda x$ of $\Phi^\lambda K$. Combining this property with \eqref{Philambda} and \eqref{L0sumsmooth} yields that
\begin{equation}
\label{Philambdaxu}
\langle v,\Phi^\lambda x  \rangle=
h_K(v)^{1-\lambda}h_{\Phi K}(v)^\lambda.
\end{equation}
On the other hand, $\nu_K(x)=u$ for $x\in\partial' K$ and
$\nu_{\Phi K}(\Phi x)=\Phi^{-1}u$ for $\Phi x\in\partial'(\Phi K)$, 
we deduce from \eqref{non-parallel} that
$$
\langle  v,x  \rangle<
h_K(v) \mbox{ \ and \ }
\langle v,\Phi x  \rangle<
h_{\Phi K}(v).
$$
In particular, the H\"older inequality yields that
\begin{eqnarray*}
\langle  v,\Phi^\lambda x  \rangle&=&\sum_{i=1}^nu_ix_i\leq
\left(\sum_{i=1}^nt_i^{-\lambda}u_ix_i\right)^{1-\lambda}
\left(\sum_{i=1}^nt_i^{1-\lambda}u_ix_i\right)^\lambda=\langle  v,x  \rangle^{1-\lambda}
\langle v,\Phi x  \rangle^\lambda\\
&<&h_K(v)^{1-\lambda}h_{\Phi K}(v)^\lambda.
\end{eqnarray*}
This contradicts \eqref{Philambdaxu}, and proves Lemma~\ref{logsumPhi}.
\end{proof}

Finally, we characterize the equality case of Theorem~\ref{B.F.M.S.}.

\begin{prop}
\label{B.F.M.S.equa}
	If $K$ and $L$ are unconditional convex bodies in $\R^n$ with respect to the same orthonormal basis and
 $\lambda\in(0,1)$, then
\begin{equation}
\label{logBMequa0}
V((1-\lambda)\cdot K +_0 \lambda\cdot L)= V(K)^{1-\lambda} V(L)^\lambda,
\end{equation}
holds if and only if $K=K_1\oplus\ldots \oplus K_m$ and $L=L_1\oplus\ldots \oplus L_m$ for unconditional compact convex sets 
	$K_1,\ldots, K_m,L_1,\ldots,L_m$ of dimension at least one, $m\geq 1$ where $\sum_{i=1}^m{\rm dim}\,K_i=n$
	and $K_i$ and $L_i$ are dilates, $i=1,\ldots,m$.
\end{prop}

\begin{proof}
On the one hand, if we have $K$ and $L$ as described after \eqref{logBMequa0},
then we have equality according to Lemma~\ref{sumlowerlogBMequa}.

On the other hand, if \eqref{logBMequa0} holds, then
we deduce from Lemma~\ref{coordinatewise-convex}
and Theorem~\ref{coordinatewise-ineq-lemma} that
that
$$
V(K^{1-\lambda}\cdot L^\lambda)= V(K)^{1-\lambda} V(L)^\lambda.
$$
According to Lemma~\ref{coordinatewise-ineq-lemma}, there exists 
a positive definite diagonal matrix $\Phi$ such that $L=\Phi K$.
Finally Lemma~\ref{logsumPhi} completes the proof of Proposition~\ref{B.F.M.S.equa}.
\end{proof}

\noindent {\bf Acknowledgement } We are grateful for helpful discussions with Franck Barthe and Martin Henk. We are deeply indebted to the referee who has corrected various mistakes and simplified various arguments in an earlier version of the paper and whose remarks have significantly improved the presentation.

\end{document}